\documentclass[11pt]{amsart}

\usepackage{amsmath,amssymb,amsthm,mathtools}
\usepackage{enumitem}
\usepackage{hyperref}
\usepackage[margin=1.05in]{geometry}

\usepackage{graphicx}

\hypersetup{colorlinks=true,linkcolor=blue,citecolor=blue,urlcolor=blue}

\newtheorem{theorem}{Theorem}[section]
\newtheorem{proposition}[theorem]{Proposition}
\newtheorem{lemma}[theorem]{Lemma}
\newtheorem{corollary}[theorem]{Corollary}
\newtheorem{remark}[theorem]{Remark}
\newtheorem*{remark*}{Remark}

\theoremstyle{definition}

\newtheorem*{definition*}{Definition}

\newcommand{\D}{\mathbb D}
\newcommand{\R}{\mathbb R}

\newcommand{\Ai}{\operatorname{Ai}}
\newcommand{\norm}[1]{\left\lVert #1\right\rVert}

\title{Exact $L^p$ growth rates of Laplace eigenfunctions on the unit disk}
\author{Haoyu Cheng}
\begin{document}

\begin{abstract}
We determine the logarithmic growth exponents of the $L^p$ norms, $1\le p\le\infty$, of $L^2$-normalized Laplace eigenfunctions on the unit disk, for both Dirichlet and Neumann boundary conditions. We also prove sharp uniform $L^p$ upper and lower bounds for every $L^2$-normalized Dirichlet eigenfunction and every non-constant Neumann eigenfunction $u_{\lambda}$ on the disk. The proof uses stationary phase estimates and integral estimates for Bessel functions.
\end{abstract}

\maketitle

\section{Introduction and main results}

Let
\[
\D=\{(x,y)\in\R^2:x^2+y^2<1\}
\]
be the unit disk. In this paper we consider $L^2$-normalized eigenfunctions of $-\Delta$ on $\D$ with either Dirichlet or Neumann boundary condition. Our first theorem gives the sharp uniform upper and lower bounds for their $L^p$ norms for all $1\le p\le\infty$.

\begin{theorem}\label{thm:uniform}
Let $1\le p\le\infty$. Then there exist constants $c_p,C_p>0$ such that for every $L^2$-normalized Dirichlet eigenfunction and every non-constant $L^2$-normalized Neumann eigenfunction $u_\lambda$ satisfying
\[
-\Delta u_{\lambda}=\lambda u_{\lambda},\qquad \norm{u_{\lambda}}_{L^2(\D)}=1,
\]
\[
u_{\lambda}\vert_{\partial\D}=0\quad(\text{Dirichlet})\qquad\text{or}\qquad
\left.\partial_\nu u_\lambda\right|_{\partial \D}=0 \quad(\text{Neumann})
\]
one has
\[
c_p\lambda^{a_p}\le \norm{u_\lambda}_{L^p(D)}\le C_p\lambda^{b_p}.
\]
For $p=\infty$,
\[
 a_\infty=\frac1{12},\qquad b_\infty=\frac14.
\]
For $1\le p<\infty$,
\[
a_p=
\begin{cases}
\dfrac{p-2}{6p},&1\le p<2,\\[8pt]
0,&2\le p\le4,\\[8pt]
\dfrac{p-4}{12p},&4< p<\infty,
\end{cases}
\qquad b_p=
\begin{cases}
0,&1\le p<2,\\[8pt]
\dfrac{p-2}{6p},&2\le p\le8,\\[8pt]
\dfrac14-\dfrac1p,&8< p<\infty.
\end{cases}
\]
The exponents $a_p$ and $b_p$ are sharp.
\end{theorem}
The upper and lower bounds, as functions of $1/p$, are shown in Figure \ref{fig:uniform-exponents}.
\begin{figure}[htbp]
\centering
\includegraphics[width=1\textwidth]{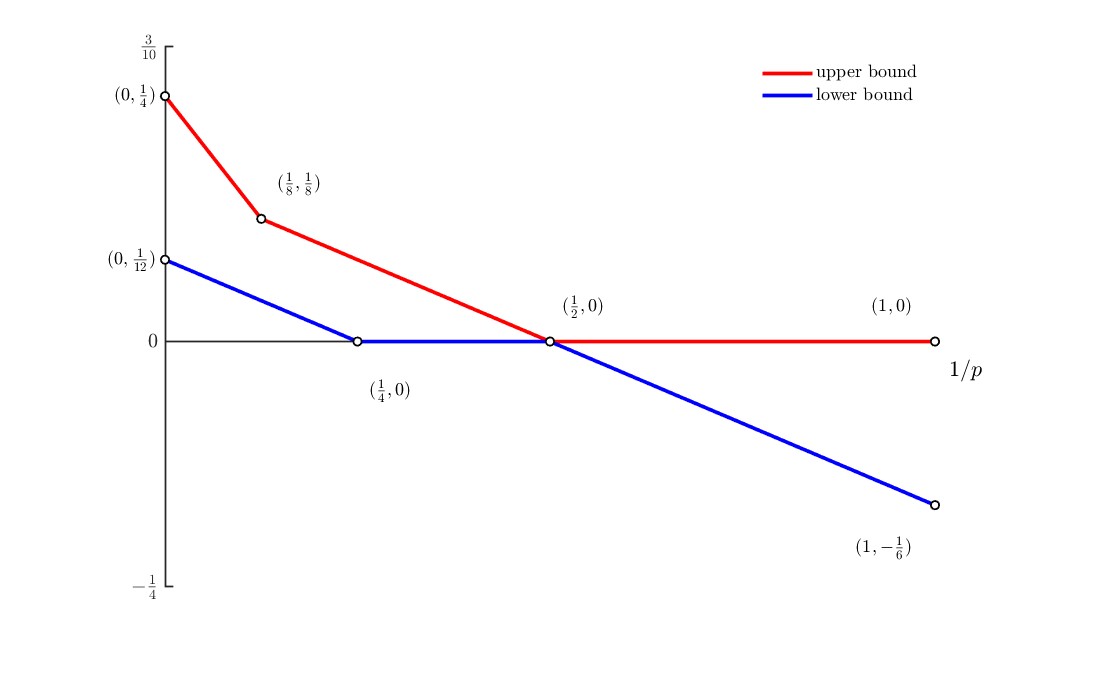}
\caption{The upper and lower exponents in Theorem \ref{thm:uniform} as functions of $1/p$. The red curve represents the upper exponent $b_p$, and the blue curve represents the lower exponent $a_p$.}
\label{fig:uniform-exponents}
\end{figure}

\begin{remark*}
The upper bounds in Theorem \ref{thm:uniform} are known. For $2\le p\le\infty$ they follow from the sharp two-dimensional estimates on compact manifolds with boundary of Smith and Sogge \cite{SmithSogge2007Boundary}, and for $1\le p<2$ the upper bound is an immediate consequence of $\norm{u_\lambda}_{L^2(\D)}=1$. The lower bound is also standard for $2\le p\le4$ by finite volume of $\D$ and the $L^2$-normalization, and for $1\le p<2$ it follows by interpolating $L^2$ between $L^p$ and the known $L^4$ upper bound. In the endpoint $p=\infty$, Lavoie and Poliquin \cite{Lavoie2020GrowthRO} proved global lower bounds of order $\lambda^{1/18-\varepsilon}$ for Dirichlet eigenfunctions and $\lambda^{1/12-\varepsilon}$ for Neumann eigenfunctions, for every $\varepsilon>0$. The present theorem improves their results to $\lambda^{1/12}$ for both boundary conditions. Thus, the new contribution is the uniform lower bounds for $4<p\le\infty$ and the sharpness of the exponents for all $1\le p\le\infty$.
\end{remark*}

\subsection{Detailed study of asymptotics of eigenfunctions}

Let $J_n$ be the Bessel function of the first kind. We denote by $j_{n,m}$ the $m$-th positive zero of $J_n$, and by $\rho_{n,m}$ the $m$-th positive zero of $J_n'$.

The disk eigenfunctions separate in polar coordinates. It is known that each eigenspace has dimension at most $2$ for both boundary conditions \cite{ashu2013some}. For the Dirichlet problem the nonzero eigenvalues are $\lambda^D_{n,m}=j_{n,m}^2$, while for the Neumann problem, excluding the constant eigenfunction, they are $\lambda^N_{n,m}=\rho_{n,m}^2$.

For $n\ge1$ we use the following normalized representatives:
\begin{equation}\label{eq:dir-eigenfunction}
u^D_{n,m}(r,\theta)=\sqrt{\frac2\pi}\,\frac{J_n(j_{n,m}r)}{|J_{n+1}(j_{n,m})|}\cos(n\theta),
\end{equation}
and
\begin{equation}\label{eq:neu-eigenfunction}
u^N_{n,m}(r,\theta)=\left(\frac{2}{\pi\left(1-\frac{n^2}{\rho_{n,m}^2}\right)}\right)^{1/2}\frac{J_n(\rho_{n,m}r)}{|J_n(\rho_{n,m})|}\cos(n\theta).
\end{equation}

For $n \ge 1$, every element of the corresponding two-dimensional eigenspace is obtained by replacing $\cos(n\theta)$ with
\[
a\cos(n\theta)+b\sin(n\theta),\qquad a,b\in\mathbb C.
\]
For each $1\le p\le\infty$, the $L^p(0,2\pi)$ norm of this angular factor is uniformly comparable to its $L^2(0,2\pi)$ norm:
\[
c_p\norm{a\cos(n\theta)+b\sin(n\theta)}_{L^2(0,2\pi)}
\le \norm{a\cos(n\theta)+b\sin(n\theta)}_{L^p(0,2\pi)}
\le C_p\norm{a\cos(n\theta)+b\sin(n\theta)}_{L^2(0,2\pi)}.
\]
Indeed, by the change of variables $\phi=n\theta$, this follows from equivalence of norms on the two-dimensional space $\operatorname{span}_{\mathbb C}\{\cos\phi,\sin\phi\}$. Hence every $L^2$-normalized eigenfunction in the same eigenspace has $L^p(\D)$ norm comparable to that of the representative above. Thus it is enough to study the representatives above.

For $n = 0$ the eigenspace is one-dimensional. Up to multiplication by a constant of modulus $1$, the normalized eigenfunctions are
\[
u^D_{0,m}(r,\theta)=\frac{1}{\sqrt{\pi}}\frac{J_0(j_{0,m}r)}{|J_1(j_{0,m})|},\qquad
u^N_{0,m}(r,\theta)=\frac{1}{\sqrt{\pi}}\frac{J_0(\rho_{0,m}r)}{|J_0(\rho_{0,m})|}.
\]

When $\lambda_{n,m}\to\infty$, we must investigate how both $m,n$ vary.

\begin{definition*}
Fix $1\le p\le\infty$, and let $B\in\{D,N\}$ be the type of boundary condition. For $\gamma\in[0,\infty)$, define $E^{B,p}_\gamma(\mathbb D)$ by
\[
E^{B,p}_\gamma(\mathbb D):=\left\{\ell\in\mathbb R:\,
\begin{array}{l}
\text{there exists a sequence }(n_k,m_k)_{k\ge1}\subset\mathbb Z_{\ge0}\times \mathbb Z_{\ge1}, \\[2pt]
\text{with } n_k\ge2,\ \lambda^B_{n_k,m_k}\to\infty,\text{ such that} \\[2pt]
\dfrac{\log m_k}{\log n_k}\to\gamma,\quad\text{and}\quad
\dfrac{\log \norm{u^B_{n_k,m_k}}_{L^p(\D)}}{\log \lambda^B_{n_k,m_k}}\to \ell
\end{array}
\right\}.
\]
For $\gamma=\infty$, we define $E^{B,p}_\infty(D)$ using sequences satisfying either $n_k\ge 2,\ \frac{\log m_k}{\log n_k}\to\infty$, or $n_k\in\{0,1\}\ ,\ m_k\to\infty$.
\end{definition*}

With this notation, the problem of determining the possible growth rates of normalized eigenfunctions is reduced to describing the limit point sets $E^{B,p}_\gamma(\mathbb D)$. The next theorem shows that, once the logarithmic relation between $m$ and $n$ is fixed, the logarithmic ratio between the $L^p$ norm and the eigenvalue has a unique limit.

\begin{theorem}\label{thm:profile}
Let $1\le p\le\infty$. For every $\gamma\in[0,\infty]$, both $E^{D,p}_\gamma(\mathbb D)$ and $E^{N,p}_\gamma(\mathbb D)$ are singleton sets, and
\[
E^{D,p}_\gamma(\mathbb D)=E^{N,p}_\gamma(\mathbb D)=\{\Phi_p(\gamma)\},
\]
where
\[
\Phi_p(\gamma):=
\begin{cases}
\dfrac{(1-\gamma)(p-2)}{6p},
&0\le \gamma<1, \quad 1\le p<4,\\[8pt]
\dfrac16-\dfrac1{3p}-\dfrac{\gamma}{12},
&0\le \gamma<1, \quad 4\le p\le\infty,\\[8pt]
0,
&1\le \gamma\le\infty, \quad 1\le p<4,\\[8pt]
\dfrac14-\dfrac1p-\dfrac1{6\gamma}+\dfrac{2}{3p\gamma},
&1\le \gamma\le\infty, \quad 4\le p\le\infty.
\end{cases}
\]
\end{theorem}
\begin{remark*}
The terms $1/\gamma$ and $1/p$ are interpreted as zero when $\gamma=\infty$ or $p=\infty$. When $p=\infty$, we confirm the conjecture in \cite{Lavoie2020GrowthRO}. The formulas agree at $p=4$ and at $\gamma=1$.
\end{remark*}
We present the graphs of $\Phi_p(\gamma)$ as a function of $1/p$ for $\gamma=0,1,\infty$ in Figure \ref{fig:profile-exponents}.
\begin{figure}[htbp]
\centering
\includegraphics[width=1\textwidth]{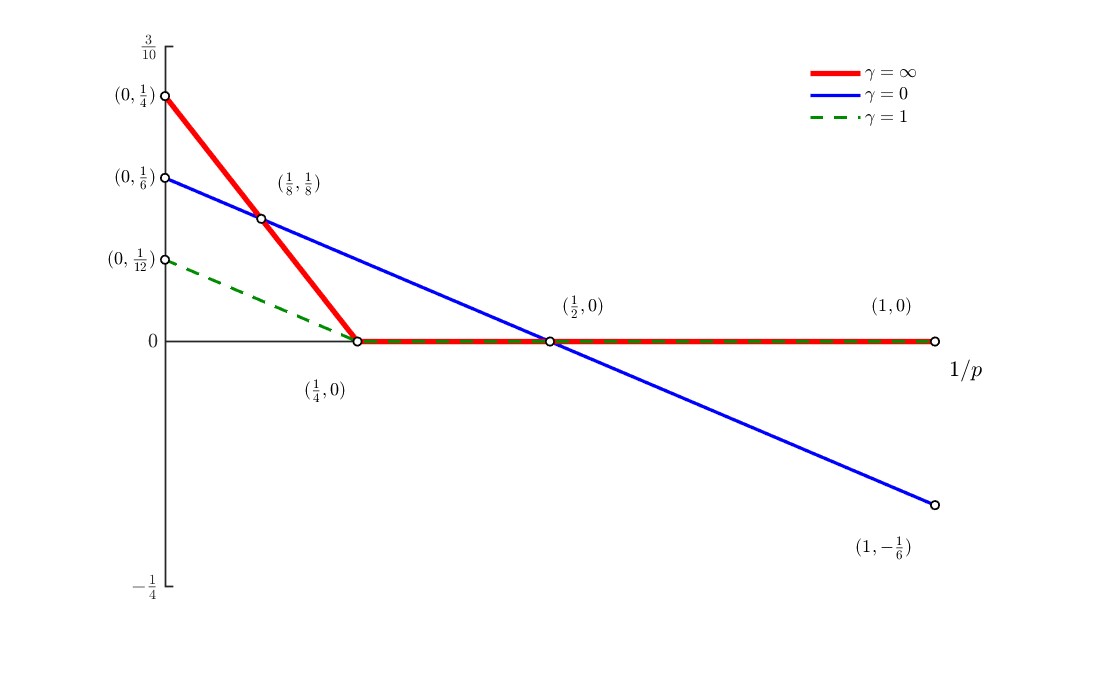}
\caption{$\Phi_p(\gamma)$ as a function of $1/p$ for three representative values of $\gamma$. The blue curve represents $\gamma=0$, the green dashed curve represents $\gamma=1$, and the red curve represents $\gamma=\infty$.}
\label{fig:profile-exponents}
\end{figure}

Let $E^{B,p}(\mathbb D)$ be the set of accumulation points, as $\lambda\to\infty$, of
\[
\frac{\log \norm{u_\lambda}_{L^p(\mathbb D)}}{\log\lambda}
\]
over all $L^2$-normalized eigenfunctions satisfying the boundary condition $B$. When $p=\infty$, this definition was introduced by Sarnak \cite{sarnak2004morawetz} on page 41 of his letter to Morawetz. Theorem \ref{thm:profile} gives the following exact form of $E^{B,p}(\mathbb D)$.

\begin{corollary}
For every $1\le p\le\infty$,
\[
E^{D,p}(\mathbb D)=E^{N,p}(\mathbb D)=[a_p,b_p],
\]
where $a_p$ and $b_p$ are the exponents in Theorem \ref{thm:uniform}.
\end{corollary}

\begin{proof}
Let a sequence of normalized eigenfunctions with eigenvalues tending to infinity be given. After passing to a subsequence, either $n_k\to\infty$ or $n_k$ is fixed. In the first case we may pass to a further subsequence so that $\log m_k/\log n_k$ converges to some $\gamma\in[0,\infty]$, and Theorem \ref{thm:profile} gives the logarithmic exponent $\Phi_p(\gamma)$. In the second case $m_k\to\infty$, and we have $\gamma=\infty$ by definition. Hence Theorem \ref{thm:profile} gives the exponent $\Phi_p(\infty)$. Thus every accumulation point belongs to the image of $\Phi_p$.

Conversely, for each finite $\gamma\ge0$ one may let $n_k\to\infty$ and take $m_k=\lfloor n_k^\gamma\rfloor$ when $\gamma>0$ and a fixed $m_k$ when $\gamma=0$. For $\gamma=\infty$ one may take $n_k=2$ fixed and let $m_k\to\infty$. Theorem \ref{thm:profile} then realizes all values in the image of $\Phi_p$.

A direct inspection of $\Phi_p$ gives the following ranges. For $p=\infty$ this range is $[1/12,1/4]$. For $1\le p\le2$ it is $[(p-2)/(6p),0]$, for $2\le p\le4$ it is $[0,(p-2)/(6p)]$, for $4<p<8$ it is $[(p-4)/(12p),(p-2)/(6p)]$, and for $8\le p<\infty$ it is $[(p-4)/(12p),1/4-1/p]$. These intervals are exactly $[a_p,b_p]$.
\end{proof}

\begin{remark*}
Let $\partial_\nu$ be the outward normal derivative. The logarithmic size of the natural boundary data is also determined by the same index relation. More precisely, let $1\le p\le\infty$ and $\gamma\in[0,\infty]$, and suppose that the sequence $(n_k,m_k)$ satisfies the index conditions associated with $\gamma$ in the above definition. In the Dirichlet case,
\[
\lim_{k\to\infty} \frac{\log \norm{\partial_\nu u^D_{n_k,m_k}}_{L^p(\partial\D)}}{\log \lambda^D_{n_k,m_k}}=\frac12.
\]
In the Neumann case,
\[
\lim_{k\to\infty} \frac{\log \norm{u^N_{n_k,m_k}}_{L^p(\partial\D)}}{\log \lambda^N_{n_k,m_k}}=
\begin{cases}
\dfrac{1-\gamma}{6},&0\le\gamma\le1,\\[6pt]
0,&1\le\gamma\le\infty.
\end{cases}
\]
These statements follow directly from the explicit formulas and estimates proved later in the paper.
\end{remark*}

\subsection{Related work}

A central theme in spectral theory is to understand how large an eigenfunction can be in terms of its eigenvalue. On a compact Riemannian manifold $(M^d,g)$ without boundary of dimension $d$, the following upper bound \cite{Hormander1968TheSF} is well known
\[
\norm{u_\lambda}_{L^\infty(M)} \le C(M,g)\lambda^{(d-1)/4}.
\]
This estimate is sharp on the round sphere $S^d$, where zonal spherical harmonics attain the exponent $(d-1)/4$. On the other hand, the trivial lower bound
\[
\norm{u_\lambda}_{L^\infty(M)} \ge \operatorname{Vol}(M)^{-1/2}
\]
is sharp on flat tori, where plane waves remain uniformly bounded. Thus, the possible $L^\infty$-growth of normalized eigenfunctions is strongly influenced by the geometry and in fact the dynamical properties of the geodesic flow.

More generally, Sogge \cite{Sogge1988Lp} proved sharp $L^p$ spectral cluster estimates on compact manifolds without boundary. In dimension two, if $-\Delta u_\lambda=\lambda u_\lambda$ and $\norm{u_\lambda}_{L^2}=1$, these estimates give
\begin{equation}\label{eq:boundaryless-estimates}
\norm{u_\lambda}_{L^p}\le C_p
\begin{cases}
\lambda^{\frac18-\frac1{4p}},&2\le p\le6,\\[4pt]
\lambda^{\frac14-\frac1p},&6\le p\le\infty.
\end{cases}
\end{equation}
For compact two-dimensional manifolds with boundary, Smith and Sogge \cite{SmithSogge2007Boundary} proved the corresponding sharp spectral cluster estimates. In the eigenvalue normalization used here, their exponents are
\[
\lambda^{\frac16-\frac1{3p}} \quad(2\le p\le8),\qquad
\lambda^{\frac14-\frac1p} \quad(8\le p\le\infty),
\]
which coincide with the upper exponents in Theorem \ref{thm:uniform} for $p\ge2$. However, this boundary loss is not present for every boundary geometry. For strictly geodesically concave boundary, Grieser \cite{Grieser1992Concave} proved the estimate \eqref{eq:boundaryless-estimates} in dimension two and locally near the concave part of the boundary for Dirichlet boundary condition. Blair, Ford and Marzuola \cite{BlairFordMarzuola2018Polygonal} proved the estimates \eqref{eq:boundaryless-estimates} for any compact planar polygonal domain and for either Dirichlet or Neumann boundary conditions.

This reflects the whispering-gallery concentration near a smooth convex boundary. In the disk this mechanism is exactly the endpoint $\gamma=0$, where the model case is that $m$ remains bounded while $n\to\infty$. The corresponding eigenfunctions concentrate in a boundary collar of thickness $n^{-2/3}\simeq \lambda^{-1/3}$. An $L^2$-normalized function essentially supported in a set of area $\lambda^{-1/3}$ has typical size $\lambda^{1/6}$, and therefore its $L^p$ norm is of size
\[
\lambda^{1/6}\bigl(\lambda^{-1/3}\bigr)^{1/p}=\lambda^{\frac16-\frac1{3p}},
\]
which is precisely the upper branch in Theorem \ref{thm:uniform}.

The most classical and most studied model is the unit sphere. On $\mathbb{S}^2$, the eigenspace of degree $N$ has dimension $2N+1$, and the eigenfunctions are spherical harmonics. This high multiplicity produces a wide range of possible behaviors. Zonal harmonics attain the maximal $L^\infty$-growth $\lambda^{1/4}$, while Gaussian beams have a different concentration mechanism along closed geodesics. For finite $p$, these two model families also account for the sharpness of the two branches in Sogge's $L^p$ estimates on the sphere. At the opposite end, probabilistic and constructive results show that one can find spherical harmonics with much smaller $L^\infty$-growth \cite{VanderKam,burq2013injections,han2025sphericalharmonicspointconfigurations}. These examples show that the sphere has a substantially wider range of possible $L^\infty$ behavior than the disk. This phenomenon is closely tied to the large dimension of the spherical harmonic eigenspaces.

Flat tori form the other basic integrable model. On $\mathbb T^2$, an eigenfunction is a trigonometric polynomial supported on lattice points lying on a circle. Thus the $L^\infty$-norm is connected to the arithmetic of lattice points on circles. In particular, the divisor bound gives
\[
\norm{u_\lambda}_{L^\infty(\mathbb T^2)} \le C_\varepsilon \lambda^\varepsilon
\]
for every $\varepsilon>0$. Bourgain \cite{bourgain1993eigenfunction} initiated the study of eigenfunction bounds on higher-dimensional tori, and more recent work of Germain--Rydin Myerson \cite{germain2022bounds} and Demeter--Germain \cite{DemeterGermain} investigates spectral projectors on thin annuli and narrow windows on tori.

The sphere and the flat torus belong to the class of manifolds with completely integrable geodesic flow. The disk is another integrable model. Unlike the sphere, however, the eigenspaces are uniformly bounded in dimension. For the Dirichlet problem, this follows from Bourget's hypothesis on the common zeros of Bessel functions, proved by Siegel. For the Neumann problem, an analogous statement for the zeros of $J_n'$ implies that Neumann eigenspaces on the disk are also at most two-dimensional \cite{watson1995treatise,ashu2013some}. Thus the disk has no high-multiplicity mechanism to create cancellations inside large eigenspaces. This is one reason why the disk is a natural model for studying lower bounds on $\norm{u_\lambda}_{L^\infty}$.

The $L^\infty$ growth-rate problem for disk eigenfunctions was studied by Lavoie and Poliquin \cite{Lavoie2020GrowthRO}. For every $\varepsilon>0$, they proved a global lower bound of order $\lambda^{1/18-\varepsilon}$ for Dirichlet eigenfunctions and of order $\lambda^{1/12-\varepsilon}$ for Neumann eigenfunctions. They also gave partial information on the corresponding $L^\infty$ exponent limit sets.

Besides the distribution of individual eigenvalues and the size of individual eigenfunctions, one may also study spectral projectors. We also mention the recent work of Chabert and de Verdi\`ere \cite{chabert2026linftynormsspectral}, who study $L^2\to L^\infty$ bounds for spectral projectors on shrinking frequency intervals on several quantum integrable surfaces, including the Euclidean disk.

On manifolds of nonpositive curvature, B\'erard \cite{berard1977wave} obtained a logarithmic improvement of the remainder in the local Weyl law, which implies a logarithmic improvement of the $L^\infty$ eigenfunction bound.

On arithmetic hyperbolic surfaces, the expected bounds for Hecke eigenfunctions are much stronger than the general Sogge estimates. In this direction one often expects, for each fixed finite $p$, bounds of the form
\[
\norm{u_\lambda}_{L^p}\le C_{p,\varepsilon}\lambda^\varepsilon
\]
for every $\varepsilon>0$ in appropriate arithmetic settings. The case $p=4$ is especially important because of its relation to the Watson--Ichino triple product formula and moments of automorphic $L$-functions. Humphries and Khan \cite{HumphriesKhan2025Lp} proved the bound $\norm{g}_{L^4}\ll_{\varepsilon}\lambda_g^{3/304+\varepsilon}$ for Hecke--Maass cusp forms on the modular surface, improving the general Sogge exponent $1/16$. More recently, Ki \cite{Ki2023L4} proved the conjectural $L^4$ bound $\norm{g}_{L^4}\ll_{\varepsilon}\lambda_g^\varepsilon$ for Hecke--Maass cusp forms.

There is also extensive literature on the spectral distribution of the disk. Although the disk has explicit eigenvalues in terms of zeros of Bessel functions, ordering this two-parameter family is subtle. Recent work of Filonov--Levitin--Polterovich--Sher \cite{filonov2023polya} proved P\'olya's conjecture for Euclidean balls in the Dirichlet case and for the unit disk in the Neumann case.

Finally, it is natural to ask whether similar results can be determined for other integrable planar domains, such as ellipses. In that case, the separated eigenfunctions are expressed in terms of Mathieu functions rather than Bessel functions, and the analysis appears substantially more complicated. Even spectral multiplicity questions for ellipses require delicate arguments. Recent work of Hillairet and Judge \cite{hillairet2024generic} proves simplicity for the spectrum of a generic ellipse. At present, an analogue of the explicit disk result for ellipses seems to be open to the best of the author's knowledge.

\subsection*{Notation and conventions}

In this paper, if $X=X_{n,m}$ and $Y=Y_{n,m}$ are nonnegative quantities depending on the indices $n,m$, we write $X\asymp Y$ if there exist constants $c,C>0$, independent of $n,m$, such that $cY\le X\le CY$ for all sufficiently large $n,m$. The notation $X\asymp_m Y$ means that the constants may depend on this fixed value of $m$, but are uniform in $n$.

To treat both boundary conditions uniformly, we use the following notation:
\[
x^D_{n,m}:=j_{n,m},\qquad x^N_{n,m}:=\rho_{n,m}.
\]
We also write
\[
A^D_{n,m}:=|J_{n+1}(j_{n,m})|,\qquad A^N_{n,m}:=|J_n(\rho_{n,m})|\left(1-\frac{n^2}{\rho_{n,m}^2}\right)^{1/2}.
\]
Thus $A^B_{n,m}$ is the factor in \eqref{eq:dir-eigenfunction} or \eqref{eq:neu-eigenfunction}. The computation of the $L^p$ norm leads to the following one-dimensional Bessel integral. For $1\le p<\infty$ and $R>0$, we define
\[
I_p(n,R):=\int_0^R |J_n(x)|^p x\,dx .
\]

\section{Properties of Bessel functions}

\subsection{Basic identities and estimates}

We shall repeatedly use the following standard Bessel identities.

\begin{lemma}
Let $a>0$.
\begin{enumerate}[label=(\roman*)]
\item If $J_n(a)=0$, then
\[
\int_0^1 J_n(ar)^2r\,dr=\frac12J_{n+1}(a)^2.
\]
\item If $J_n'(a)=0$, then
\[
\int_0^1 J_n(ar)^2r\,dr=\frac12\left(1-\frac{n^2}{a^2}\right)J_n(a)^2.
\]
\end{enumerate}
\end{lemma}

\begin{proof}
The recurrence formulas for Bessel functions \cite[\S 10.6(i),(10.6.2)]{NIST:DLMF} are
\[
J_{n-1}(x)=J_n'(x)+\frac nx J_n(x),\qquad J_{n+1}(x)=\frac nx J_n(x)-J_n'(x).
\]
If $J_n(a)=0$, $J_{n-1}(a)=J_n'(a)=-J_{n+1}(a)$. The two identities then follow directly from \cite[\S 10.22(ii), (10.22.37)--(10.22.38)]{NIST:DLMF}.
\end{proof}

For $n\ge1$ we define
\[
M_n(R):=\max_{0\le x\le R}|J_n(x)|.
\]

\begin{lemma}\label{lem:landau}
There exist absolute constants $c,C>0$ such that for all $n\ge1$ and all $R>n$,
\[
cn^{-1/3}\le M_n(R)\le Cn^{-1/3}.
\]
\end{lemma}

\begin{proof}
The lower bound follows from the Cauchy asymptotic formula, see Watson \cite[Ch. VIII]{watson1995treatise},
\[
J_n(n)=c_0n^{-1/3}+O(n^{-1}),\qquad c_0>0,
\]
since $n\in[0,R]$. The upper bound is Landau's inequality, see Landau \cite{Landau2000BesselFM},
\[
\sup_{x\ge0}|J_n(x)|\le Cn^{-1/3}.
\]
\end{proof}

We introduce the following elementary integration by parts estimate, which we will use repeatedly.

\begin{lemma}\label{lem:elementary-ibp}
Let $I$ be a compact interval and let $\psi\in C^2(I)$ be real-valued. Assume that $\psi'$ does not vanish on $I$.
\begin{enumerate}[label=(\roman*)]
\item If $\psi'$ is monotone and $|\psi'|\ge M>0$ on $I$, then for every $\lambda>0$,
\[
\left|\int_I e^{i\lambda\psi(t)}\,dt\right|\le \frac{C}{\lambda M}.
\]
\item More generally, if $q\in C^1(I)$ is real valued and $q/\psi'$ is monotone on $I$ with $|q/\psi'|\le G$, then for every $\lambda>0$,
\[
\left|\int_I q(t)e^{i\lambda\psi(t)}\,dt\right|\le \frac{C G}{\lambda}.
\]
\end{enumerate}
\end{lemma}

\begin{proof}
Write $I=[\alpha,\beta]$. For (ii), since $q/\psi'$ is monotone and $|q/\psi'|\le G$, integration by parts gives
\[
\int_\alpha^\beta q(t)e^{i\lambda\psi(t)}\,dt
=\left[\frac{q(t)e^{i\lambda\psi(t)}}{i\lambda\psi'(t)}\right]_\alpha^\beta
-\frac1{i\lambda}\int_\alpha^\beta e^{i\lambda\psi(t)}\frac{d}{dt}\left(\frac{q(t)}{\psi'(t)}\right)\,dt.
\]
Hence
\[
\left|\int_\alpha^\beta q(t)e^{i\lambda\psi(t)}\,dt\right|
\le \frac{2G}{\lambda}+\frac1\lambda\int_\alpha^\beta \left|\frac{d}{dt}\left(\frac{q(t)}{\psi'(t)}\right)\right|\,dt
=\frac{2G}{\lambda}+\frac1\lambda \left|\frac{q(\beta)}{\psi'(\beta)}-\frac{q(\alpha)}{\psi'(\alpha)}\right|
\le \frac{4G}{\lambda}.
\]
Part (i) is the special case $q\equiv1$ and $G=M^{-1}$, because $1/\psi'$ is monotone when $\psi'$ is monotone and nonvanishing.
\end{proof}

\begin{lemma}\label{lem:pointwise-bessel}
For all integers $n\ge0$ and all $x>n$,
\[
|J_n(x)|\le C(x^2-n^2)^{-1/4}.
\]
\end{lemma}

\begin{proof}
Let $\kappa:=\sqrt{x^2-n^2}$. Bessel's integral representation \cite[\S 10.9(i), (10.9.2)]{NIST:DLMF} gives
\[
J_n(x)=\frac1\pi\int_0^\pi \cos(x\sin t-nt)\,dt.
\]
The phase $\phi(t)=x\sin t-nt$ has a unique stationary point in $(0,\pi)$, namely $t_0=\arccos(n/x)$, and $0<t_0\le\pi/2$. Moreover,
\[
|\phi''(t_0)|=x\sin t_0=\kappa.
\]
Let $\delta=\kappa^{-1/2}$, and $I_0=[t_0-\delta,t_0+\delta]\cap[0,\pi]$. The contribution of $I_0$ is bounded by its length, hence is $O(\kappa^{-1/2})$.

It remains to estimate the contribution of $[0,\pi]\setminus I_0$. On this set we have $|t-t_0|\ge\delta$. Moreover,
\[
|\phi'(t)|=2x\sin\frac{t+t_0}2\left|\sin\frac{t-t_0}2\right|\ge \frac4\pi x\sin\frac{t_0}2\left|\frac{t-t_0}2\right|\ge c\kappa |t-t_0|.
\]
In particular $|\phi'(t)|\ge c\kappa^{1/2}$ whenever $|t-t_0|\ge\delta$. The complement $[0,\pi]\setminus I_0$ has at most two connected components. On each of them, $\phi'(t)=x\cos t-n$ is monotone and does not vanish. Applying Lemma \ref{lem:elementary-ibp} with $\lambda=1$ and $M=c\kappa^{1/2}$ gives a contribution $O(\kappa^{-1/2})$ from each component. Therefore $|J_n(x)|\le C\kappa^{-1/2}=C(x^2-n^2)^{-1/4}$.
\end{proof}

\subsection{Uniform asymptotic estimates}
\label{subsec:uniform-estimate}

To analyze the asymptotics of $J_n$, we introduce the following notations. For $x>n$ define $\beta\in(0,\pi/2)$ by $x=\frac n{\cos\beta}$. Set
\[
\Theta_n(x):=n(\tan\beta-\beta)=\sqrt{x^2-n^2}-n\arccos\frac n x,\qquad
\Lambda_n(x):=n\frac{\sin^3\beta}{\cos\beta}.
\]
The following proposition is the central asymptotic formula. It is proved by a stationary phase argument from Bessel's integral representation.

\begin{proposition}\label{prop:osc-asymptotic}
There exist absolute constants $K_0,C>0$ such that, whenever $n\ge1$ and $x>n$ satisfy
\[
x-n\ge K_0 n^{1/3},
\]
the following estimates hold:
\begin{equation}\label{eq:bessel-asymptotic}
J_n(x)=\sqrt{\frac2\pi}(x^2-n^2)^{-1/4}\left(\cos\left(\Theta_n(x)-\frac\pi4\right)+O(\Lambda_n(x)^{-1/10})\right),
\end{equation}
\begin{equation}\label{eq:bessel-derivative-asymptotic}
J_n'(x)=-\sqrt{\frac2\pi}\frac{(x^2-n^2)^{1/4}}{x}\left(\sin\left(\Theta_n(x)-\frac\pi4\right)+O(\Lambda_n(x)^{-1/10})\right).
\end{equation}
The constants in the two $O$-terms are absolute. Moreover, for every $\varepsilon>0$ there exists $K_\varepsilon\ge K_0$ such that both error terms have absolute value at most $\varepsilon$ whenever $x-n\ge K_\varepsilon n^{1/3}$. In particular, the error terms tend to zero uniformly for every family satisfying $(x-n)/n^{1/3}\to\infty$.
\end{proposition}

\begin{proof}
Bessel's integral representation gives
\[
J_n(x)=\frac1{2\pi}\int_{-\pi}^{\pi}e^{i(x\sin t-nt)}\,dt
=\Re\left(\frac1{\pi}\int_{0}^{\pi}e^{i(x\sin t-nt)}\,dt\right).
\]
Writing $x=n\sec\beta$, define
\[
\phi_\beta(t):=\sec\beta\sin t-t.
\]
The phase $\phi_\beta$ has stationary point $t=\beta$, and
\[
\phi_\beta(\beta)=\tan\beta-\beta,\qquad\phi_\beta''(\beta)=-\tan\beta.
\]

Let
\[
\kappa:=\tan\beta,\qquad \omega:=\Lambda_n(x)^{1/10},\qquad \delta:=\frac{\omega}{\sqrt{n\kappa}}.
\]

When $\frac{x-n}{n^{1/3}}\ge1$, 
\[
1-\frac nx\ge\frac{n^{1/3}}{n+n^{1/3}}\ge\frac12n^{-2/3},
\]
therefore
\[
\Lambda_n(x)=x\cdot\left(1-\frac{n^2}{x^2}\right)^{\frac32}\ge x\cdot\left(1-\frac{n}{x}\right)^{\frac32}=\sqrt{(1-\frac nx)(x-n)^2}\ge \frac{\sqrt{2}}2\frac{x-n}{n^{1/3}}.
\]

Thus, after increasing $K_0$ if necessary, $\Lambda_n(x)$ is sufficiently large throughout the range of the proposition. Moreover
\[
\frac{\delta}{\sin\beta}=\frac{\omega}{\sqrt{n\sin^3\beta/\cos\beta}}=\Lambda_n(x)^{-2/5}.
\]
By taking $K_0$ large enough we may therefore assume $\delta<\beta/2$, and the neighborhood $[\beta-\delta,\beta+\delta]$ is contained in $(0,\pi)$.

On each component of the complement of this neighborhood, $\phi_\beta'$ is monotone and does not vanish. Moreover, on $[0,\beta-\delta]$,
\[
|\phi_\beta'(t)|\ge \frac{\cos(\beta-\delta)-\cos\beta}{\cos\beta}\ge \frac{\delta\sin(\beta-\delta)}{\cos\beta}\ge \frac{1}{2}\delta\tan\beta=\frac12\kappa\delta,
\]
and the same lower bound holds on $[\beta+\delta,\pi]$. Lemma \ref{lem:elementary-ibp}, applied with $\lambda=n$ and $M=c\kappa\delta$, gives the contribution away from $\beta$ as
\[
O\left(\frac1{n\kappa\delta}\right)=O\left((n\kappa)^{-1/2}\Lambda_n(x)^{-1/10}\right).
\]

Near $\beta$, Taylor expansion gives
\[
\phi_\beta(\beta+s)=\phi_\beta(\beta)-\frac\kappa2s^2+O(\sec\beta |s|^3).
\]
Using
\[
n(\sec\beta)|s|^3\le n(\sec\beta)\delta^3=O(\Lambda_n(x)^{-1/5})\qquad (|s|\le\delta)
\]
and $|e^{ix}-1|\le |x|$, the error made in replacing $e^{in\phi_\beta(\beta+s)}$ by
\[
e^{in\phi_\beta(\beta)}e^{-\frac{i}{2}n\kappa s^2}=e^{i\Theta_n(x)}e^{-\frac{i}{2}n\kappa s^2}
\]
in the integral over $[\beta-\delta,\beta+\delta]$ is bounded by
\[
O\left(\delta\Lambda_n(x)^{-1/5}\right)=O\left((n\kappa)^{-1/2}\Lambda_n(x)^{-1/10}\right).
\]
Consequently,
\[
\int_{\beta-\delta}^{\beta+\delta}e^{in\phi_\beta(t)}\,dt
=e^{i\Theta_n(x)}\int_{-\delta}^\delta e^{-\frac{1}{2}in\kappa s^2}\,ds+O\left((n\kappa)^{-1/2}\Lambda_n(x)^{-1/10}\right).
\]
Using
\[
\int_{-\infty}^\infty e^{-\frac{i}{2}s^2}\,ds=\sqrt{2\pi}e^{-i\pi/4}\qquad\text{and}\qquad \int_{\omega}^\infty e^{-\frac{i}{2}s^2}\,ds=O\left(\frac1\omega\right),
\]
the change of variables $s\mapsto s/\sqrt{n\kappa}$ gives
\[
\int_{\beta-\delta}^{\beta+\delta}e^{in\phi_\beta(t)}\,dt
=e^{i\Theta_n(x)-i\pi/4}\sqrt{\frac{2\pi}{n\kappa}}+O\left((n\kappa)^{-1/2}\Lambda_n(x)^{-1/10}\right).
\]
Combining the estimates and taking real parts yields
\[
J_n(x)=\sqrt{\frac{2}{\pi n\kappa}}\left(\cos\left(\Theta_n(x)-\frac\pi4\right)+O(\Lambda_n(x)^{-1/10})\right).
\]
Since $x^2-n^2=n^2\tan^2\beta=n^2\kappa^2$, this is \eqref{eq:bessel-asymptotic}.

Differentiating Bessel's integral representation gives
\[
J_n'(x)=\frac{i}{2\pi}\int_{-\pi}^{\pi}(\sin t)\,e^{i(x\sin t-nt)}\,dt
=\Re\left(\frac{i}{\pi}\int_{0}^{\pi}(\sin t)\,e^{i(x\sin t-nt)}\,dt\right).
\]
For the derivative asymptotic we use the weighted part of Lemma \ref{lem:elementary-ibp}. On the two complementary intervals, the quotient
\[
\frac{\sin t}{\phi_\beta'(t)}=\frac{\sin t\cos\beta}{\cos t-\cos\beta}
\]
is monotone on each of the two intervals. Indeed,
\[
\frac{d}{dt}\left(\frac{\sin t\cos\beta}{\cos t-\cos\beta}\right)=\frac{\cos\beta(1-\cos\beta\cos t)}{(\cos t-\cos\beta)^2}>0, \qquad t\ne\beta.
\]
For the relevant endpoints $t=\beta\pm\delta,0,\pi$ one has
\[
\left|\frac{\sin t}{\phi_\beta'(t)}\right|\le\frac{4\cos\beta}{\delta}.
\]
By monotonicity, the same bound holds for all $t$ in both intervals. Thus Lemma \ref{lem:elementary-ibp}, with $\lambda=n$ and $G=C\cos\beta/\delta$, gives the contribution away from $\beta$ as
\[
O\left(\frac{\cos\beta}{n\delta}\right)=O\left(\frac{\sin\beta}{\sqrt{n\kappa}}\Lambda_n(x)^{-1/10}\right).
\]

Near $\beta$, we use $\sin t=\sin\beta+O(t-\beta)$ and the preceding stationary phase estimate:
\[
\begin{aligned}
\int_{\beta-\delta}^{\beta+\delta}(\sin t)\,e^{in\phi_\beta(t)}\,dt
&=\sin\beta\int_{\beta-\delta}^{\beta+\delta}e^{in\phi_\beta(t)}\,dt+\int_{\beta-\delta}^{\beta+\delta}O(t-\beta)\,dt\\
&=\sin\beta\sqrt{\frac{2\pi}{n\kappa}}e^{i\Theta_n(x)-i\pi/4}+O\left(\frac{\sin\beta}{\sqrt{n\kappa}}\Lambda_n(x)^{-1/10}\right)+O(\delta^2).
\end{aligned}
\]
Since
\[
\delta^2=\frac{\sin\beta}{\sqrt{n\kappa}}\Lambda_n(x)^{-3/10},
\]
we finally get
\[
\int_0^\pi(\sin t)\,e^{in\phi_\beta(t)}\,dt
=\sin\beta\sqrt{\frac{2\pi}{n\kappa}}e^{i\Theta_n(x)-i\pi/4}+O\left(\frac{\sin\beta}{\sqrt{n\kappa}}\Lambda_n(x)^{-1/10}\right).
\]
Together with
\[
\frac{\sin\beta}{\pi}\sqrt{\frac{2\pi}{n\kappa}}=\sqrt{\frac2\pi}\frac{(x^2-n^2)^{1/4}}{x},
\]
this gives \eqref{eq:bessel-derivative-asymptotic}.
\end{proof}

The oscillatory estimates above do not cover the endpoint $\gamma=0$ when $m$ does not tend to $\infty$. We need another lemma to deal with this case.

\begin{lemma}\label{lem:fixed-turning}
Let $m_0\in\mathbb{N}$ be fixed. We denote by $a_m$ the absolute value of the $m$-th negative zero of the Airy function $\Ai$, and by $a_m'$ the absolute value of the $m$-th negative zero of $\Ai'$. Then
\[
j_{n,m_0}=n+2^{-1/3}a_{m_0}n^{1/3}+o(n^{1/3}),
\]
\[
\rho_{n,m_0}=n+2^{-1/3}a'_{m_0}n^{1/3}+o(n^{1/3}),
\]
and
\[
J_n'(j_{n,m_0})=-2^{2/3}n^{-2/3}\operatorname{Ai}'(-a_{m_0})+o(n^{-2/3}),
\]
\[
J_n(\rho_{n,m_0})=2^{1/3}n^{-1/3}\operatorname{Ai}(-a'_{m_0})+o(n^{-1/3}).
\]
Consequently, for either boundary condition $B\in\{D,N\}$ and fixed $m$,
\[
A^B_{n,m}\asymp_m n^{-2/3},\qquad \norm{u^B_{n,m}}_{L^\infty(\D)}\asymp_m n^{1/3}.
\]
\end{lemma}

\begin{proof}
The two zero asymptotics follow from the proof of Proposition 3.2.1 in \cite{Lavoie2020GrowthRO}, see also \cite{Qu1999BestPU}. The asymptotics for the values at the zeros and derivative zeros follow directly from the large-order asymptotic formulas for the associated values of $J_\nu$, see \cite[\S 10.21(viii), Eqs.~(10.21.42) and (10.21.44)]{NIST:DLMF}, and also \cite{royal1960bessel}.

It remains to justify the last assertion. The asymptotic for $\rho_{n,m}$ gives $\rho_{n,m}^2-n^2\asymp_m n^{4/3}$. Hence in the Neumann case
\[
\left(1-\frac{n^2}{\rho_{n,m}^2}\right)^{1/2}\asymp_m n^{-1/3}.
\]
The asymptotics for $J_n'(j_{n,m})$ and $J_n(\rho_{n,m})$ therefore give $A^B_{n,m}\asymp_m n^{-2/3}$ for both boundary conditions. Moreover Lemma \ref{lem:landau} gives
\[
\max_{0\le x\le j_{n,m}} |J_n(x)|\asymp_m n^{-1/3},\qquad
\max_{0\le x\le \rho_{n,m}} |J_n(x)|\asymp_m n^{-1/3},
\]
since by the preceding zero asymptotics, both zeros are larger than $n$ for all large $n$. Dividing this maximal size by the corresponding normalizing factor in \eqref{eq:dir-eigenfunction} or \eqref{eq:neu-eigenfunction} gives $\norm{u^B_{n,m}}_{L^\infty(\D)}\asymp_m n^{1/3}$.
\end{proof}

\subsection{Estimates for zeros of Bessel functions}

The following results give estimates for the zeros $j_{n,m}$ and $\rho_{n,m}$. The next lemma is due to Lavoie and Poliquin \cite{Lavoie2020GrowthRO}.

\begin{lemma}\label{lem:zeros-near}
There is an absolute integer $M_*$ such that, whenever $n\ge1$ and $m\ge M_*$, the following estimates hold:
\[
n+C_1m^{2/3}n^{1/3}<j_{n,m}<n+C_1m^{2/3}n^{1/3}+C_2m^{4/3}n^{-1/3},
\]
\[
n+C_1(m-1)^{2/3}n^{1/3}<\rho_{n,m}<n+C_1m^{2/3}n^{1/3}+C_2m^{4/3}n^{-1/3},
\]
where $C_1=\frac{(9\pi^2)^{1/3}}2$ and $C_2=\frac{9(3\pi^4)^{1/3}}{40}$.
\end{lemma}

\begin{proof}
See \cite{Lavoie2020GrowthRO}, Proposition 3.2.1 and Proposition 3.2.2.
\end{proof}

We shall use the following elementary consequences of standard estimates for Bessel zeros.

\begin{lemma}\label{lem:zeros-global}
For every integer $m\ge1$ and $n\ge0$, the following holds:
\[
m\pi-\frac{\pi}4\le j_{n,m}\le m\pi+\frac\pi 2n,
\]
\[
m\pi-2\pi\le\rho_{n,m}\le m\pi+\frac\pi2(n+1).
\]
\end{lemma}

\begin{proof}
From \cite[\S 10.21(iv), Eq.~(10.21.17)]{NIST:DLMF} we know
\[
\frac{\partial j_{\nu,m}}{\partial\nu}=2j_{\nu,m}\int_0^\infty K_0(2j_{\nu,m}\sinh t)e^{-2\nu t}\,dt,
\]
where $K_\nu$ is the modified Bessel function of the second kind and $K_0(y)=\int_0^\infty e^{-y\cosh s}\,ds$. Hence
\[
0\le \frac{\partial j_{\nu,m}}{\partial\nu}\le2j_{\nu,m}\int_0^\infty K_0(2j_{\nu,m}\sinh t)\,dt
=2j_{\nu,m}\int_0^\infty \frac{K_0(y)}{\sqrt{4j_{\nu,m}^2+y^2}}\,dy \le\int_0^\infty K_0(y)\,dy=\frac{\pi}2.
\]
Therefore
\[
|j_{\nu,m}-j_{\mu,m}|\le \frac\pi2|\nu-\mu|.
\]
Since $J_{1/2}(x)=\sqrt{2/(\pi x)}\sin x$, we have $j_{1/2,m}=m\pi$. If $n\ge1$, then the monotonicity in $\nu$ gives $j_{n,m}\ge j_{1/2,m}=m\pi$, while the Lipschitz estimate gives
\[
j_{n,m}\le m\pi+\frac\pi2\left(n-\frac12\right)\le m\pi+\frac\pi2 n.
\]
If $n=0$, the upper bound follows from the same monotonicity, $j_{0,m}\le j_{1/2,m}=m\pi$, while $j_{0,m}\ge j_{1/2,m}-\pi/4= m\pi-\pi/4$. This proves the stated bounds for $j_{n,m}$.

For the derivative zeros, if $n\ge1$ the standard interlacing $j_{n,m-1}<\rho_{n,m}<j_{n,m}$ holds with the convention $j_{n,0}=0$. Hence $\rho_{n,m}\le j_{n,m}\le m\pi+\frac\pi2n$. For the lower bound, the case $m=1$ is immediate from positivity, while for $m\ge2$ the estimate for $j_{n,m-1}$ gives
\[
\rho_{n,m}>j_{n,m-1}\ge (m-1)\pi-\frac\pi4\ge m\pi-2\pi.
\]
If $n=0$, then $\rho_{0,m}=j_{1,m}$ because $J_0'=-J_1$, and the same lower bound and the asserted upper bound follow from the estimate for $j_{1,m}$. This proves the lemma.
\end{proof}

\subsection{Fixed-order estimates}

The fixed-order estimates are used when $n$ is bounded and $m$ tends to infinity.

\begin{lemma}\label{lem:fixed-order-estimates}
Fix an integer $N\ge0$. Uniformly for $0\le n\le N$, $B\in\{D,N\}$ and $m\to\infty$, one has
\[
x^B_{n,m}\asymp_N m,\qquad A^B_{n,m}\asymp_N (x^B_{n,m})^{-1/2},
\]
and
\[
\norm{u^B_{n,m}}_{L^\infty(\D)}\asymp_N (x^B_{n,m})^{1/2}.
\]
Moreover, for every $1\le p<\infty$,
\[
I_p(n,x^B_{n,m})\asymp_{p,N}
\begin{cases}
(x^B_{n,m})^{2-p/2},&1\le p<4,\\[4pt]
1+\log x^B_{n,m},&p=4,\\[4pt]
1,&p>4.
\end{cases}
\]
\end{lemma}

\begin{proof}
We use the asymptotic expansions for large argument from \cite[\S 10.17 (10.17.3)(10.17.9)]{NIST:DLMF}. Taking the first term in these asymptotic expansions gives the estimates used below, uniformly for $0\le n\le N$ as $x\to\infty$.
\[
J_n(x)=\left(\frac2{\pi x}\right)^{1/2}\left(\cos(x-\frac{n\pi}{2}-\frac\pi4)+O_N(x^{-1})\right),
\]
\[
J_n'(x)=-\left(\frac2{\pi x}\right)^{1/2}\left(\sin(x-\frac{n\pi}{2}-\frac\pi4)+O_N(x^{-1})\right).
\]
Unlike Proposition \ref{prop:osc-asymptotic}, these fixed-order expansions include the case $n=0$, which is needed here.

Lemma \ref{lem:zeros-global} gives $x^B_{n,m}\asymp_N m$ for $0\le n\le N$. We next estimate the normalizing factors. $J_n(j_{n,m})=0$ gives $\cos(j_{n,m}-\frac{n\pi}{2}-\frac\pi4)=O_N(j_{n,m}^{-1})$, and therefore $|\sin(j_{n,m}-\frac{n\pi}{2}-\frac\pi4)|\ge 1/2$ for all large $m$. We get $|J_n'(j_{n,m})|\asymp_N j_{n,m}^{-1/2}$. Since $J_{n+1}(j_{n,m})=-J_n'(j_{n,m})$, this proves the Dirichlet estimate for $A^D_{n,m}$. Similar arguments give $|J_n(\rho_{n,m})|\asymp_N\rho_{n,m}^{-1/2}$. Since $n\le N$ and $\rho_{n,m}\to\infty$, $\left(1-\frac{n^2}{\rho_{n,m}^2}\right)^{1/2}\asymp_N1$, the Neumann estimate for $A^N_{n,m}$ follows.

We now prove the $L^\infty$ estimate. The fixed-order asymptotic above gives a uniform upper bound for $J_n(x)$ when $x$ is large, while continuity gives a uniform upper bound on compact intervals for $0\le n\le N$. For the lower bound, there are finitely many $n$ and no $J_n$ is identically zero. Thus $\max_{0\le x\le x^B_{n,m}} |J_n(x)|\asymp_N1$. Dividing by the factor $A^B_{n,m}$ gives
\[
\norm{u^B_{n,m}}_{L^\infty(\D)}\asymp_N (x^B_{n,m})^{1/2}.
\]

It remains to estimate $I_p$. From the asymptotic above, $|J_n(x)|\le C_N x^{-1/2}$ for $x\ge1$. Since $J_n$ is also bounded on $[0,1]$, this gives
\[
I_p(n,R)\le C_{p,N}
\begin{cases}
R^{2-p/2},&1\le p<4,\\[2pt]
1+\log R,&p=4,\\[2pt]
1,&p>4.
\end{cases}
\]
For the matching lower bounds, choose $X_N$ so large that the error term in the asymptotic for $J_n$ is at most $1/4$ for every $0\le n\le N$ and $x\ge X_N$. On the set where $|\cos(x-\frac{n\pi}{2}-\frac\pi4)|\ge1/2$, we then have $|J_n(x)|\ge c_N x^{-1/2}$. Each interval of length $2\pi$ contains a subinterval of uniformly positive length on which this condition holds. Summing over the complete intervals contained in $[X_N,R]$ gives for all sufficiently large $R$
\[
I_p(n,R)\ge c_{p,N}
\begin{cases}
R^{2-p/2},&1\le p<4,\\[2pt]
1+\log R,&p=4,\\[2pt]
1,&p>4.
\end{cases}
\]
Applying these estimates with $R=x^B_{n,m}$ proves the desired estimate for $I_p(n,x^B_{n,m})$.
\end{proof}

\section{The case \texorpdfstring{$p=\infty$}{p=infinity}}

\subsection{The asymptotic proof of Theorem \ref{thm:profile} for \texorpdfstring{$p=\infty$}{p=infinity}}

\begin{lemma}\label{lem:normalizing-factors}
There exists $M_0\in\mathbb N$ such that
\[
A^B_{n,m}\asymp \frac{\big((x^B_{n,m})^2-n^2\big)^{1/4}}{x^B_{n,m}}
\]
uniformly for $m\ge M_0$, $n\ge0$ and $B\in\{D,N\}$.
\end{lemma}

\begin{proof}
The case $n=0$ follows from Lemma \ref{lem:fixed-order-estimates}, since $A^B_{0,m}\asymp (x^B_{0,m})^{-1/2}$ and the right hand side is exactly $(x^B_{0,m})^{-1/2}$ when $n=0$.

Assume $n\ge1$, choose $M_0$ large enough such that $x^B_{n,m}$ satisfies the condition $x^B_{n,m}-n\ge K_0n^{1/3}$ in Proposition \ref{prop:osc-asymptotic} whenever $m\ge M_0$, and the two error terms there have absolute value at most $1/10$. This is possible by the zero estimates of Bessel functions already proved.

In the Dirichlet case, $J_n(j_{n,m})=0$ and \eqref{eq:bessel-asymptotic} implies $|\cos(\Theta_n(j_{n,m})-\frac\pi4)|\le \frac1{10}$. Hence $|\sin(\Theta_n(j_{n,m})-\frac\pi4)|\ge \frac12$. Using \eqref{eq:bessel-derivative-asymptotic} and the recurrence relation $J_{n+1}(j_{n,m})=-J_n'(j_{n,m})$ we obtain
\[
A^D_{n,m}=|J_{n+1}(j_{n,m})|\asymp \frac{(j_{n,m}^2-n^2)^{1/4}}{j_{n,m}}.
\]

In the Neumann case, $J_n'(\rho_{n,m})=0$ and \eqref{eq:bessel-derivative-asymptotic} implies $|\sin(\Theta_n(\rho_{n,m})-\frac\pi4)|\le \frac1{10}$, and therefore $|\cos(\Theta_n(\rho_{n,m})-\frac\pi4)|\ge \frac12$. Thus \eqref{eq:bessel-asymptotic} gives $|J_n(\rho_{n,m})|\asymp(\rho_{n,m}^2-n^2\big)^{-1/4}$. Multiplying by 
\[
\left(1-\frac{n^2}{\rho_{n,m}^2}\right)^{1/2}=\frac{(\rho_{n,m}^2-n^2)^{1/2}}{\rho_{n,m}}
\]
gives
\[
A^N_{n,m}\asymp \frac{(\rho_{n,m}^2-n^2)^{1/4}}{\rho_{n,m}}.
\]
This proves the lemma.
\end{proof}

\begin{proof}[Proof of Theorem \ref{thm:profile} for $p=\infty$]
It suffices to verify the asserted limit after passing to subsequences. If $n$ is bounded and $m\to\infty$, then after passing to a further subsequence $n$ is fixed, and Lemma \ref{lem:fixed-order-estimates} gives
\[
\norm{u^B_{n,m}}_{L^\infty(\D)}\asymp_n (x^B_{n,m})^{1/2},\qquad \lambda^B_{n,m}=(x^B_{n,m})^2.
\]
Thus the logarithmic exponent is $1/4=\Phi_\infty(\infty)$.

It remains to consider the case $n\to\infty$. In all cases relevant to the logarithmic limits, $m\to\infty$ holds except possibly when $\gamma=0$ and $m$ stays bounded. If $m\to\infty$, Lemma \ref{lem:normalizing-factors} and Lemma \ref{lem:landau} give the common formula
\begin{equation}\label{eq:linfty-common}
\norm{u^B_{n,m}}_{L^\infty(\D)}\asymp n^{-1/3}\frac{x^B_{n,m}}{\big((x^B_{n,m})^2-n^2\big)^{1/4}}.
\end{equation}
This single formula applies to both boundary conditions.

\subsubsection{\textbf{The range $0\le\gamma<1$}}
Any sequence of $(n,m)$ has a subsequence satisfying either $m\to\infty$ or $m$ is eventually constant. So we just need to prove these two cases. 

If $m\to\infty$. Assume $\frac{\log m}{\log n}\to\gamma\in[0,1)$. Then $m=o(n)$. Lemma \ref{lem:zeros-near} gives
\[
x^B_{n,m}=n+O(n^{1/3}m^{2/3})=n(1+o(1)),
\]
and
\[
(x^B_{n,m})^2-n^2\asymp n^{4/3}m^{2/3}.
\]
Substituting into \eqref{eq:linfty-common},
\[
\norm{u^B_{n,m}}_{L^\infty(\D)}\asymp n^{-1/3}\frac{n}{(n^{4/3}m^{2/3})^{1/4}}=n^{1/3}m^{-1/6}.
\]
Since
\[
\lambda^B_{n,m}=(x^B_{n,m})^2=n^2(1+o(1)),
\]
and $m=n^{\gamma+o(1)}$, we obtain
\[
\frac{\log \norm{u^B_{n,m}}_{L^\infty(\D)}}{\log\lambda^B_{n,m}}\to\frac{\frac13-\frac\gamma6}{2}=\frac{2-\gamma}{12}.
\]

If $m$ is fixed, we must have $\gamma=0$. Lemma \ref{lem:fixed-turning} gives directly
\[
\norm{u^B_{n,m}}_{L^\infty(\D)}\asymp_m n^{1/3},\qquad \lambda^B_{n,m}\asymp_m n^2,
\]
so the limit is $1/6=\Phi_\infty(0)$.

\subsubsection{\textbf{The range $1\le\gamma<\infty$}}
Let
\[
\frac{\log m}{\log n}\to\gamma\ge1.
\]
The same formula \eqref{eq:linfty-common} holds. We only need to compute logarithmic sizes of $x^B_{n,m}$ and $(x^B_{n,m})^2-n^2$.

If $\gamma=1$, then $m=n^{1+o(1)}$. Lemma \ref{lem:zeros-global} gives $x^B_{n,m}=n^{1+o(1)}$. On the other hand, Lemma \ref{lem:zeros-near} gives
\[
x^B_{n,m}-n\ge c n^{1/3}(m-1)^{2/3}=n^{1+o(1)}.
\]
Hence
\[
(x^B_{n,m})^2-n^2=(x^B_{n,m}-n)(x^B_{n,m}+n)=n^{2+o(1)}.
\]
Thus \eqref{eq:linfty-common} gives
\[
\norm{u^B_{n,m}}_{L^\infty(\D)}=n^{1/6+o(1)},\qquad \lambda^B_{n,m}=n^{2+o(1)},
\]
and hence the limit of $\frac{\log \norm{u^B_{n,m}}_{L^\infty(\D)}}{\log\lambda^B_{n,m}}$ is $1/12$.

\begin{remark}\label{rem:lower-sharp-linfty}
In particular, when $m=n\to\infty$, Lemma \ref{lem:zeros-near} gives
\[
x^B_{n,m}\asymp n,\qquad (x^B_{n,m})^2-n^2\asymp n^2.
\]
Thus \eqref{eq:linfty-common} gives
\[
\norm{u^B_{n,m}}_{L^\infty(\D)}\asymp n^{1/6},\qquad \lambda^B_{n,m}\asymp n^2,
\]
which shows that the lower exponent in the $p=\infty$ part of Theorem \ref{thm:uniform} is sharp.
\end{remark}

If $\gamma>1$, then $m=n^{\gamma+o(1)}$ and Lemma \ref{lem:zeros-global} gives $x^B_{n,m}=n^{\gamma+o(1)}$. In particular $x^B_{n,m}\gg n$, and therefore
\[
(x^B_{n,m})^2-n^2=(x^B_{n,m})^2(1+o(1))=n^{2\gamma+o(1)}.
\]
Substituting into \eqref{eq:linfty-common},
\[
\norm{u^B_{n,m}}_{L^\infty(\D)}=n^{-1/3}\frac{n^{\gamma+o(1)}}{n^{\gamma/2+o(1)}}=n^{\gamma/2-1/3+o(1)}.
\]
Also
\[
\lambda^B_{n,m}=n^{2\gamma+o(1)}.
\]
Therefore
\[
\frac{\log \norm{u^B_{n,m}}_{L^\infty(\D)}}{\log\lambda^B_{n,m}}\to\frac{\frac\gamma2-\frac13}{2\gamma}=\frac14-\frac1{6\gamma}.
\]

\subsubsection{\textbf{The endpoint $\gamma=\infty$}}
The case $n$ stays bounded was handled at the start of the proof. We now consider $n\to\infty$ and $\frac{\log m}{\log n}\to\infty$. Since $m\to\infty$, Lemma \ref{lem:zeros-global} gives
\[
\pi m-O(1)\le x^B_{n,m}\le C(m+n).
\]
Because $\log n=o(\log m)$ in the present case, this implies $\log x^B_{n,m}=\log m+o(\log m)$ and, in particular, $n/x^B_{n,m}\to0$. Consequently
\[
(x^B_{n,m})^2-n^2=(x^B_{n,m})^2\left(1-\frac{n^2}{(x^B_{n,m})^2}\right)=(x^B_{n,m})^2(1+o(1)).
\]
Thus the common formula \eqref{eq:linfty-common} becomes
\[
\norm{u^B_{n,m}}_{L^\infty(\D)}\asymp n^{-1/3}\frac{x^B_{n,m}}{\big((x^B_{n,m})^2-n^2\big)^{1/4}}\asymp n^{-1/3}(x^B_{n,m})^{1/2}.
\]
Taking logarithms, we get
\[
\log \norm{u^B_{n,m}}_{L^\infty(\D)}=\frac12\log x^B_{n,m}-\frac13\log n+O(1)=\left(\frac12+o(1)\right)\log x^B_{n,m}=\left(\frac14+o(1)\right)\log \lambda^B_{n,m}.
\]
Hence the limit is $1/4$. This completes the proof of Theorem \ref{thm:profile} in the case $p=\infty$.
\end{proof}

\begin{remark}\label{rem:upper-sharp-linfty}
In particular, when $n=n_0$ is fixed and $m\to\infty$, Lemma \ref{lem:fixed-order-estimates} gives
\[
\norm{u^B_{n,m}}_{L^\infty(\D)}\asymp (x^B_{n,m})^{1/2},\qquad \lambda^B_{n,m}=(x^B_{n,m})^2,
\]
which shows that the upper exponent in the $p=\infty$ part of Theorem \ref{thm:uniform} is sharp.
\end{remark}

\subsection{The uniform \texorpdfstring{$L^\infty$}{L-infinity} estimates}

\begin{proof}[Proof of Theorem \ref{thm:uniform} for $p=\infty$]
The finitely many eigenfunctions below any fixed eigenvalue threshold may be absorbed into the constants. Therefore we may assume that at least one of $n,m$ is sufficiently large.

First consider $n=0$. Lemma \ref{lem:fixed-order-estimates} gives
\[
\norm{u^B_{0,m}}_{L^\infty(\D)}\asymp (x^B_{0,m})^{1/2},\qquad \lambda^B_{0,m}=(x^B_{0,m})^2.
\]
Thus
\[
c(\lambda_{0,m}^B)^{1/12}\le c(\lambda_{0,m}^B)^{1/4}\le \norm{u^B_{0,m}}_{L^\infty(\D)}\le C(\lambda_{0,m}^B)^{1/4}.
\]

Now let $n\ge1$. Choose an integer $M$ at least as large as the thresholds in Lemmas \ref{lem:normalizing-factors} and \ref{lem:zeros-near}. If $1\le m<M$, Lemma \ref{lem:fixed-turning} gives, after taking the maximum over this finite set of $m$ and absorbing finitely many small values of $n$,
\[
\norm{u^B_{n,m}}_{L^\infty(\D)}\asymp n^{1/3},\qquad \lambda^B_{n,m}\asymp n^2.
\]
This gives the desired bounds in the range $1\le m<M$.

It remains to treat $m\ge M$. By Lemma \ref{lem:normalizing-factors} and Lemma \ref{lem:landau}, the common formula \eqref{eq:linfty-common} holds. If $m\le n$, Lemma \ref{lem:zeros-near} gives
\[
x^B_{n,m}\asymp n,\qquad (x^B_{n,m})^2-n^2\asymp n^{4/3}m^{2/3},
\]
and therefore
\[
\norm{u^B_{n,m}}_{L^\infty(\D)}\asymp n^{1/3}m^{-1/6}.
\]
Since $1\le m\le n$ and $x^B_{n,m}\asymp n$, this implies
\[
\norm{u^B_{n,m}}_{L^\infty(\D)}\ge c (x^B_{n,m})^{1/6} \qquad\text{and}\qquad
\norm{u^B_{n,m}}_{L^\infty(\D)}\le C (x^B_{n,m})^{1/2}.
\]

If $m\ge n$, Lemma \ref{lem:zeros-global} gives $x^B_{n,m}\asymp m$ and $(x^B_{n,m})^2-n^2\asymp (x^B_{n,m})^2$. Hence
\[
\norm{u^B_{n,m}}_{L^\infty(\D)}\asymp n^{-1/3}(x^B_{n,m})^{1/2}.
\]
Since $x^B_{n,m}\ge n$, this is bounded below by $c(x^B_{n,m})^{1/6}$ and above by $C(x^B_{n,m})^{1/2}$. Combining the cases and using $\lambda^B_{n,m}=(x^B_{n,m})^2$ proves
\[
c(\lambda_{n,m}^B)^{1/12}\le \norm{u^B_{n,m}}_{L^\infty(\D)}\le C(\lambda_{n,m}^B)^{1/4}.
\]
The sharpness follows from Remarks \ref{rem:lower-sharp-linfty} and \ref{rem:upper-sharp-linfty}.
\end{proof}

\section{Integral estimates for Bessel functions}
\label{sec:bessel-integrals}

For the $L^p$ norm, after changing variables the one-dimensional Bessel integral $I_p(n,R)$ introduced in the notation and conventions part appears naturally. According to the size of $R$ relative to $n$, we shall need three asymptotic estimates. The first concerns the scale $R-n=O(n^{1/3})$. The second concerns the range $0<R-n\le n$, and the third concerns the range $R-n\ge c n$. We discuss them in the next three subsections.

\subsection{The scale \texorpdfstring{$R-n=O(n^{1/3})$}{R-n=O(n1/3)}}

\begin{lemma}\label{lem:below-turning-lp}
Let $1\le p<\infty$.
\begin{enumerate}[label=(\roman*)]
\item For every fixed $A>0$ there is a constant $C_{p,A}$ such that, whenever $n\ge1$ and $0<R\le n+A n^{1/3}$,
\[
I_p(n,R)\le C_{p,A} n^{4/3-p/3}.
\]
\item For every fixed $\delta>0$ there is a constant $c_{p,\delta}>0$ such that, whenever $n\ge1$ and $R\ge n+\delta n^{1/3}$,
\[
I_p(n,R)\ge c_{p,\delta}n^{4/3-p/3}.
\]
\end{enumerate}
\end{lemma}

\begin{proof}
We first prove the upper bound. We shall use Bessel’s integral representation and give a pointwise estimate for $x\le n$. Write
\[
x=n-n^{1/3}T,\qquad 0\le T\le n^{2/3}.
\]
We claim that, for every fixed integer $N\ge1$, there exists $C_N$ such that
\[
|J_n(n-n^{1/3}T)|\le C_N n^{-1/3}(1+T)^{-N},\qquad 0\le T\le n^{2/3}.
\]
By Bessel's integral representation,
\[
J_n(x)=\frac1{2\pi}\int_{-\pi}^{\pi}e^{i(x\sin t-nt)}\,dt.
\]

We first treat the range $0\le T\le1$. In this case it is enough to prove $|J_n(n-n^{1/3}T)|\le Cn^{-1/3}$. Let $\phi(t)=x\sin t-nt, x=n-n^{1/3}T$, and $\delta=n^{-1/3}$. On $[-\delta,\delta]$ we use only the length of the interval $\left|\int_{-\delta}^{\delta}e^{i\phi(t)}\,dt\right| \le 2\delta =2n^{-1/3}$. We now estimate the integral over $[\delta,\pi]$. Since $\phi'(t)=x\cos t-n$, and $x\le n$, for $t\in[\delta,\pi/2]$ we have
\[
|\phi'(t)|=n-x\cos t\ge n(1-\cos t)\ge c n t^2\ge c n^{1/3}.
\]
For $t\in[\pi/2,\pi]$, since $\cos t\le0$ we also have $|\phi'(t)|\ge n\ge n^{1/3}$. Therefore $|\phi'(t)|\ge c n^{1/3}$ for $t\in[\delta,\pi]$. Moreover, $\phi''(t)=-x\sin t\le0$ on $[\delta,\pi]$. Applying Lemma \ref{lem:elementary-ibp} with $\lambda=1$ and $M=cn^{1/3}$, we get $\left|\int_{\delta}^{\pi}e^{i\phi(t)}\,dt\right| \le Cn^{-1/3}$. The same argument applies to $[-\pi,-\delta]$. We obtain the desired bound for $0\le T\le1$.

It remains to consider $T\ge1$. Let $\chi\in C_c^\infty((-2\eta,2\eta))$ be $1$ on $[-\eta,\eta]$, where $\eta>0$ is fixed sufficiently small. We split the integration representation above into the part with the factor $1-\chi(t)$ and the part with the factor $\chi(t)$.

On the support of $1-\chi$ we have
\[
|x\cos t-n|\ge c_\eta n,\qquad 0\le x\le n.
\]
Since the phase and amplitude are $2\pi$-periodic, repeated integration by parts on the circle gives no boundary terms. More precisely, with
\[
\mathcal L_\phi^* f:=-\frac{d}{dt}\left(\frac{f(t)}{i\phi'(t)}\right),
\]
one has
\[
\int_{-\pi}^{\pi}(1-\chi(t))e^{i\phi(t)}\,dt=\int_{-\pi}^{\pi}(\mathcal L_\phi^*)^r(1-\chi)(t)e^{i\phi(t)}\,dt .
\]
On this support $|\phi'|\ge c_\eta n$ and $|\phi^{(k)}|\le C_{k,\eta}n$ for $k\ge2$, hence an induction gives $| (\mathcal L_\phi^*)^r(1-\chi)(t)|\le C_{r,\eta}n^{-r}$. This part is therefore $O_{N,\eta}(n^{-N})$, which is smaller than $C_Nn^{-1/3}T^{-N}$ since $T\le n^{2/3}$.

For the localized part let $t=n^{-1/3}u$. It becomes
\[
n^{-1/3}\int_{-n^{1/3}\pi}^{n^{1/3}\pi} e^{i\Psi_{n,T}(u)}a(u)\,du,\qquad a(u)=\chi(n^{-1/3}u),
\]
where
\[
\Psi_{n,T}(u)=(n-n^{1/3}T)\sin(n^{-1/3}u)-n^{2/3}u .
\]
We have
\[
\Psi_{n,T}'(u)=-T\cos(n^{-1/3}u)-n^{2/3}(1-\cos(n^{-1/3}u))=-T\cos(n^{-1/3}u)-2n^{2/3}\sin^2\left(\frac{n^{-1/3}u}2\right),
\]
\[
\Psi_{n,T}''(u)=-(n^{1/3}-n^{-1/3}T)\sin(n^{-1/3}u).
\]
On the support of $a$, after choosing $\eta$ small enough, $|n^{-1/3}u|\le2\eta$ and
\[
|\Psi_{n,T}'(u)|\ge c(T+u^2),\qquad |\Psi_{n,T}''(u)|\le C|u|,\qquad |\Psi_{n,T}^{(k)}(u)|\le C_k\quad(k\ge3),
\]
and all derivatives of $a$ are bounded uniformly in $n$ and $T$. Define
\[
\mathcal L_\Psi^* f:=-\frac{d}{du}\left(\frac{f(u)}{i\Psi_{n,T}'(u)}\right).
\]
We claim that, for every $r\ge1$,
\[
|(\mathcal L_\Psi^*)^r a(u)|\le C_r(T+u^2)^{-r}.
\]
Indeed, this follows by induction. The case $r=0$ is clear. When $\mathcal L_\Psi^*$ is applied to a term already obtained, differentiating a numerator either differentiates $a$ or changes one factor $\Psi_{n,T}^{(j)}$ into $\Psi_{n,T}^{(j+1)}$, while differentiating a factor $(\Psi_{n,T}')^{-q}$ introduces one factor $\Psi_{n,T}''$ and one extra power of $(\Psi_{n,T}')^{-1}$. Hence, after $r$ applications, each term is a bounded linear combination of products of derivatives of $a$ and of $\Psi_{n,T}^{(j)}$, $j\ge2$, divided by at least $r$ powers of $\Psi_{n,T}'$, with one additional denominator power for each phase derivative appearing in the numerator. Since $T\ge1$, the bounds above imply $|\Psi_{n,T}^{(j)}|\le C_j(T+u^2)$ for every $j\ge2$. The extra denominator powers therefore absorb these phase derivatives and leave the factor $(T+u^2)^{-r}$.

Taking $r=N+1$ and integrating by parts gives
\[
\left|\int_{-n^{1/3}\pi}^{n^{1/3}\pi} e^{i\Psi_{n,T}(u)}a(u)\,du\right|\le C_N\int_{-\infty}^{\infty}(T+u^2)^{-N-1}\,du\le C_NT^{-N}.
\]
After multiplying by $n^{-1/3}$ and adding the contribution of $1-\chi$, we obtain the claimed bound for $T\ge1$.

We now use the claim in the integral. With $x=n-n^{1/3}T$ we have
\[
\int_0^n |J_n(x)|^p x\,dx\le n\int_0^{n^{2/3}} |J_n(n-n^{1/3}T)|^p n^{1/3}\,dT\le C_{p,N} n^{4/3-p/3}\int_0^\infty (1+T)^{-Np}\,dT.
\]
Taking $N\ge2$ so that $Np\ge2$ gives
\[
\int_0^n |J_n(x)|^p x\,dx\le C_p n^{4/3-p/3}.
\]
If $n<R\le n+A n^{1/3}$, the part above $n$ is bounded by Landau's inequality:
\[
\int_n^R |J_n(x)|^p x\,dx\le C A n^{1/3}\, n\, n^{-p/3}= C_A n^{4/3-p/3}.
\]
This proves the asserted upper bound for $I_p(n,R)$.

We now prove the lower bound. By the Cauchy asymptotic used in Lemma \ref{lem:landau},
\[
J_n(n)=c_1 n^{-1/3}+O(n^{-1}),\qquad c_1>0.
\]
We shall also need the estimate $J_n'(n)=O(n^{-2/3})$. Indeed,
\[
J_n'(n)=\frac{i}{2\pi}\int_{-\pi}^{\pi}(\sin t)\,e^{i(n\sin t-nt)}\,dt
=\Re\left(\frac{i}{\pi}\int_{0}^{\pi}(\sin t)\,e^{i(n\sin t-nt)}\,dt\right).
\]
Let $\psi(t)=t-\sin t$, and split the integral at $t=n^{-1/3}$. On $[0,n^{-1/3}]$ the absolute value is bounded by $\int_0^{n^{-1/3}}t\,dt=O(n^{-2/3})$. On $[n^{-1/3},\pi]$, the quotient $\frac{\sin t}{\psi'(t)}=\cot\frac t2$ is monotone and is bounded by $Cn^{1/3}$. Lemma \ref{lem:elementary-ibp} applied with amplitude $q(t)=\sin t$ and $\lambda=n$ gives the contribution of this interval as $O(n^{-2/3})$. This proves the estimate.

Fix a positive number $\delta\le1$ and let
\[
I_\delta:=[n,n+\delta n^{1/3}],\qquad Y_\delta:=\sup_{x\in I_\delta}|J_n'(x)|.
\]
For $x\in I_\delta$ we have
\[
0\le 1-\frac{n^2}{x^2}=\frac{(x-n)(x+n)}{x^2}\le C\delta n^{-2/3}.
\]
The Bessel equation
\[
J_n''(x)+x^{-1}J_n'(x)+\left(1-\frac{n^2}{x^2}\right)J_n(x)=0
\]
together with Landau's bound from Lemma \ref{lem:landau} gives for $x\in I_\delta$,
\[
|J_n''(x)|\le Cn^{-1}Y_\delta+C\delta n^{-2/3} n^{-1/3}= Cn^{-1}Y_\delta+C\delta n^{-1}.
\]
For every $x\in I_\delta$,
\[
|J_n'(x)|\le |J_n'(n)|+(x-n)\sup_{x\in I_\delta}|J_n''(x)|.
\]
Taking the supremum over $I_\delta$,
\[
Y_\delta\le Cn^{-2/3}+\delta n^{1/3}\big(Cn^{-1}Y_\delta+C\delta n^{-1}\big).
\]
For all sufficiently large $n$, the coefficient of $Y_\delta$ on the right is less than $1/2$, uniformly for $0<\delta\le1$. Thus
\[
Y_\delta\le Cn^{-2/3}.
\]
Consequently,
\[
|J_n(x)-J_n(n)|\le \delta n^{1/3}Y_\delta\le C\delta n^{-1/3}.
\]
Choose $\delta_0\in(0,1]$ so small that $C\delta_0<c_1/4$. Then, for every fixed $0<\delta\le\delta_0$,
\[
|J_n(x)|\ge c n^{-1/3},\qquad n\le x\le n+\delta n^{1/3},
\]
for $n$ sufficiently large with a constant $c>0$ independent of $n$. Hence, whenever $R\ge n+\delta n^{1/3}$,
\[
I_p(n,R)\ge \int_n^{n+\delta n^{1/3}} |J_n(x)|^p x\,dx\ge c_{p,\delta}n^{4/3-p/3}.
\]
This proves the lower bound for $0<\delta\le\delta_0$ and all sufficiently large $n$. Since $I_p(n,R)$ is increasing in $R$, this gives the claimed estimate for every fixed $\delta>0$. After decreasing the lower constant if necessary, the finitely many remaining values of $n$ are also covered, because for each fixed $n$ the integral $I_p(n,n+\delta n^{1/3})$ is positive.
\end{proof}

\subsection{The range \texorpdfstring{$0<R-n\le n$}{R-n=O(n)}}

We shall also use the following elementary averaging observation.

\begin{lemma}\label{lem:phase-average}
Let $1\le p<\infty$. Let $\Theta\subset\mathbb R$ be an interval containing a closed subinterval of length $2\pi$. Suppose that $w$ is positive on $\Theta$ and varies by at most a fixed factor $C_0$ on every subinterval of $\Theta$ of length at most $4\pi$. If $|E(\theta)|\le1/10$ on $\Theta$, then
\[
\int_\Theta |\cos\theta+E(\theta)|^p w(\theta)\,d\theta\ge c_{p,C_0}\int_\Theta w(\theta)\,d\theta .
\]
\end{lemma}

\begin{proof}
Decompose $\Theta$ into intervals of length $2\pi$, together with at most two end intervals of length less than $2\pi$. For each complete interval $Q$ let
\[
G_Q=\{\theta\in Q:|\cos\theta|\ge1/2\}.
\]
On $G_Q$ we have $|\cos\theta+E(\theta)|\ge2/5$. Since $w$ varies by at most $C_0$ on $Q$,
\[
\int_Q |\cos\theta+E(\theta)|^p w(\theta)\,d\theta
\ge\left(\frac25\right)^p\int_{G_Q}w(\theta)\,d\theta \ge c_{p,C_0}\int_Q w(\theta)\,d\theta .
\]
The complete intervals therefore control their own $w$-mass. Each end interval is adjacent to a complete interval, and the union of the two has length less than $4\pi$. Hence the $w$-mass of the end interval is bounded by a constant multiple of the $w$-mass of that neighboring complete interval. Summing over the complete intervals proves the lemma.
\end{proof}

\begin{lemma}\label{lem:oscillatory-average}
Let $1\le p<\infty$ and $n\ge1$. If $0<S\le n$, then
\[
I_p(n,n+S)\le C_p n^{4/3-p/3}\left(1+\int_0^{S/n^{1/3}}(1+T)^{-p/4}\,dT\right).
\]
If in addition $S\ge \sigma_0 n^{1/3}$ for a fixed $\sigma_0>0$, then
\[
I_p(n,n+S)\asymp_{p,\sigma_0}n^{4/3-p/3}\left(1+\int_0^{S/n^{1/3}}(1+T)^{-p/4}\,dT\right).
\]
The constants may depend on $p$ and on the fixed lower constant $\sigma_0$, but not on $n$ or $S$.
\end{lemma}

\begin{proof}
For the upper bound, Lemma \ref{lem:below-turning-lp}(i) applied with $A=1$ gives
\[
I_p(n,n+\min(S,n^{1/3}))\le C_p n^{4/3-p/3}.
\]
If $S\le n^{1/3}$, this already proves the desired estimate. We may therefore assume $S>n^{1/3}$. For $s>0$, Lemma \ref{lem:pointwise-bessel} gives
\[
|J_n(n+s)|\le C\big((n+s)^2-n^2\big)^{-1/4}\le C(ns)^{-1/4}.
\]
Hence
\[
\int_{n+n^{1/3}}^{n+S}|J_n(x)|^p x\,dx\le C_p\int_{n^{1/3}}^S n^{1-p/4}s^{-p/4}\,ds= C_p n^{4/3-p/3}\int_1^{S/n^{1/3}}T^{-p/4}\,dT.
\]
Combining the two estimates proves the desired upper bound for $I_p(n,n+S)$.

We now prove the lower bound. Let $\tau_0=\min\{\sigma_0,1\}$. By the lower estimate in Lemma \ref{lem:below-turning-lp}, the part $0\le x\le n+\tau_0n^{1/3}$ contributes $c_{p,\sigma_0}n^{4/3-p/3}$. This supplies the first term on the right hand side. It remains to prove the lower bound for the integration part.

Choose $K\ge2$ large enough so that the error term in Proposition \ref{prop:osc-asymptotic} is bounded by $1/10$ whenever $x-n\ge K n^{1/3}$. Let $a=n+K n^{1/3}, b=n+S$. If $b\le a$, then $S/n^{1/3}\le K$, and the right hand side in the asserted lower bound is bounded by a constant multiple of $n^{4/3-p/3}$. The initial contribution $I_p(n,n+\tau_0 n^{1/3})$ already proves the estimate. Hence assume $b>a$.

On $[a,b]$ Proposition \ref{prop:osc-asymptotic} gives
\[
J_n(x)=\sqrt{\frac2\pi}(x^2-n^2)^{-1/4}\big(\cos\vartheta(x)+E(x)\big),
\]
where $\vartheta(x)=\sqrt{x^2-n^2}-n\arccos\frac n x-\frac\pi4$ and $|E(x)|\le1/10$. Moreover $\vartheta'(x)=\frac{\sqrt{x^2-n^2}}{x}>0$. Let $q(x)=x(x^2-n^2)^{-p/4}$. We shall estimate the integral
\[
\int_a^b |\cos\vartheta(x)+E(x)|^p q(x)\,dx.
\]

It is convenient to change variables from $x$ to $\theta=\vartheta(x)$. Let $x=x(\theta)$ be the inverse map and $w(\theta)=\frac{q(x(\theta))}{\vartheta'(x(\theta))}$. Thus
\[
\int_a^b |\cos\vartheta(x)+E(x)|^p q(x)\,dx=\int_{\vartheta(a)}^{\vartheta(b)}|\cos\theta+E(x(\theta))|^p w(\theta)\,d\theta.
\]
We first investigate the variation of $w$. Write $x=n+s$. On $[a,b]$ we have $K n^{1/3}\le s\le n$, and
\[
w(\vartheta(x))=\frac{q(x)}{\vartheta'(x)}=x^2(x^2-n^2)^{-(p+2)/4}.
\]
Hence
\[
\left|\frac{d}{dx}\log w(\vartheta(x))\right|\le\frac2x+C_p\frac x{x^2-n^2}\le C_p s^{-1}.
\]
Since
\[
\frac{dx}{d\theta}=\frac{x}{\sqrt{x^2-n^2}}\le C\left(\frac ns\right)^{1/2},
\]
we get
\[
\left|\frac{d}{d\theta}\log w(\theta)\right|\le C_p n^{1/2}s^{-3/2}\le C_p K^{-3/2}.
\]
Thus $w$ varies by at most the fixed factor $\exp(4\pi C_pK^{-3/2})$ on every $\theta$-interval of length at most $4\pi$.

Let $\Theta=[\vartheta(a),\vartheta(b)]$. If $|\Theta|<2\pi$, from
\[
\vartheta'(x)\ge cK^{1/2}n^{-1/3}\qquad (a\le x\le b)
\]
we know $b-a\le C_K n^{1/3}$. In this case $S/n^{1/3}\le C_K$, so the right hand side in the lower bound is again bounded by a constant multiple of $n^{4/3-p/3}$. Thus we may assume that $\Theta$ contains at least one complete interval of length $2\pi$.

By Lemma \ref{lem:phase-average}, applied to $w$ and to $E(x(\theta))$ on $\Theta$, we get
\[
\int_\Theta |\cos\theta+E(x(\theta))|^p w(\theta)\,d\theta\ge c_p\int_\Theta w(\theta)\,d\theta.
\]
Changing variables back from $\theta$ to $x$, we obtain
\[
\int_a^b |J_n(x)|^p x\,dx\ge c_p\int_a^b x(x^2-n^2)^{-p/4}\,dx.
\]
For $x=n+s$, $K n^{1/3}\le s\le S\le n$, we have $x\ge n$ and $x^2-n^2\le 3ns$. Therefore
\[
\int_a^b x(x^2-n^2)^{-p/4}\,dx\ge c_p\int_{Kn^{1/3}}^S n^{1-p/4}s^{-p/4}\,ds=c_pn^{4/3-p/3}\int_K^{S/n^{1/3}}T^{-p/4}\,dT.
\]
Together with the first contribution from the initial part $0\le x\le n+\tau_0n^{1/3}$, this proves the lower bound.
\end{proof}

\subsection{The range \texorpdfstring{$R-n\ge c n$}{R-n>cn}}

\begin{lemma}\label{lem:far-lp-integral}
Let $1\le p<\infty$ and $n\ge1$. Suppose $R-n\ge c n$ for some fixed $c>0$. Then
\[
I_p(n,R)\asymp_{p,c}
\begin{cases}
 R^{2-p/2},&1\le p<4,\\[4pt]
 1+\log R,&p=4,\\[4pt]
 n^{4/3-p/3},&p>4.
\end{cases}
\]
\end{lemma}

\begin{proof}
Choose $K$ large enough such that the error term in Proposition \ref{prop:osc-asymptotic} is bounded by $1/10$ whenever $x-n\ge K n^{1/3}$. The part $0\le x\le n+K n^{1/3}$ contributes $\asymp_p n^{4/3-p/3}$ by Lemma \ref{lem:below-turning-lp}.

On the interval $[n+K n^{1/3},R]$, Proposition \ref{prop:osc-asymptotic} gives
\[
J_n(x)=\sqrt{\frac2\pi}(x^2-n^2)^{-1/4}(\cos\vartheta(x)+E(x)),
\]
with $|E(x)|\le1/10$ and $\vartheta'(x)=\sqrt{x^2-n^2}/x>0$. Set $q(x)=x(x^2-n^2)^{-p/4}$. The upper bound
\[
\int_{n+K n^{1/3}}^R |J_n(x)|^p x\,dx\le C_p\int_{n+K n^{1/3}}^R q(x)\,dx
\]
follows directly from this formula. For the reverse inequality, we use the same change of variables $\theta=\vartheta(x)$ and let $w(\theta)=\frac{q(x(\theta))}{\vartheta'(x(\theta))}$. When $x-n\le n$, the slow variation of $w$ was verified in the proof of Lemma \ref{lem:oscillatory-average}. When $x-n\ge n$, we have $d(\log q)/dx=O_p(1/x)$ and $\vartheta'$ is bounded above and below by positive constants, so $w$ again varies by only a fixed factor on intervals of length at most $4\pi$. Finally, $R-n\ge cn$ implies that $[\vartheta(n+Kn^{1/3}),\vartheta(R)]$ has length at least $2\pi$ for all sufficiently large $n$. The remaining finitely many $n$ are absorbed into the constants. Lemma \ref{lem:phase-average} gives
\[
\begin{aligned}
\int_{n+K n^{1/3}}^R |J_n(x)|^p x\,dx &=\int_{\vartheta(n+Kn^{1/3})}^{\vartheta(R)} |\cos\theta+E(x(\theta))|^p w(\theta)\,d\theta \\
&\ge c_p\int_{\vartheta(n+Kn^{1/3})}^{\vartheta(R)} w(\theta)\,d\theta= c_p\int_{n+K n^{1/3}}^R q(x)\,dx .
\end{aligned}
\]
Consequently
\[
\int_{n+K n^{1/3}}^R |J_n(x)|^p x\,dx\asymp_p \int_{n+K n^{1/3}}^R x(x^2-n^2)^{-p/4}\,dx.
\]
Here $R<n+Kn^{1/3}$ can occur only for finitely many $n$, and these values are absorbed into the constants.

With the substitution $u=x^2-n^2$, the last integral is a constant multiple of
\[
\int_{Kn^{1/3}(2n+Kn^{1/3})}^{R^2-n^2}u^{-p/4}\,du.
\]
If $1\le p<4$, the upper endpoint dominates. Since $R-n\ge cn$, one has $R^2-n^2\asymp_c R^2$, and the contribution is $\asymp_{p,c} R^{2-p/2}$. This also dominates the initial contribution $n^{4/3-p/3}$. If $p=4$, the same computation gives a logarithm,
\[
1+\log\frac{R^2-n^2}{n^{4/3}}\asymp_c 1+\log R.
\]
If $p>4$, the lower endpoint dominates, and this part is $\asymp_p n^{4/3-p/3}$. This proves the lemma.
\end{proof}

\section{The case \texorpdfstring{$1\le p<\infty$}{1<=p<infinity}}

In this section we prove the $L^p$ norm parts of Theorems \ref{thm:uniform} and \ref{thm:profile}.

\begin{lemma}\label{lem:lp-normalization}
Let $1\le p<\infty$ and $n\ge1$. For either boundary condition $B\in\{D,N\}$,
\[
\norm{u^B_{n,m}}_{L^p(\mathbb D)}^p\asymp_p (A^B_{n,m})^{-p}(x^B_{n,m})^{-2}I_p(n,x^B_{n,m}).
\]
The same formula holds for $n=0$, with a different constant depending only on $p$.
\end{lemma}

\begin{proof}
We prove the Dirichlet case. The Neumann case is identical after replacing the normalizing factor by the one in \eqref{eq:neu-eigenfunction}. From \eqref{eq:dir-eigenfunction},
\[
\norm{u^D_{n,m}}_{L^p(\mathbb D)}^p=\left(\frac2\pi\right)^{p/2}|J_{n+1}(j_{n,m})|^{-p}\int_0^{2\pi}|\cos(n\theta)|^p\,d\theta \int_0^1 |J_n(j_{n,m}r)|^p r\,dr .
\]
For $n\ge1$, the angular integral is a positive constant depending only on $p$:
\[
\int_0^{2\pi}|\cos(n\theta)|^p\,d\theta=\int_0^{2\pi}|\cos\theta|^p\,d\theta.
\]
The change of variables $x=j_{n,m}r$ gives
\[
\int_0^1 |J_n(j_{n,m}r)|^p r\,dr=j_{n,m}^{-2}\int_0^{j_{n,m}}|J_n(x)|^p x\,dx.
\]
This proves the lemma for Dirichlet eigenfunctions. For Neumann eigenfunctions one uses \eqref{eq:neu-eigenfunction}, and for $n=0$ the angular factor is constant which changes only the implicit constant.
\end{proof}

We now apply the integral estimates from Section \ref{sec:bessel-integrals} at the zeros of $J_n$ and $J_n'$. This is the next proposition.

\begin{proposition}\label{prop:lp-bessel-estimates}
Let $1\le p<\infty$ and $B\in\{D,N\}$. The following estimates hold uniformly in the indicated ranges.
\begin{enumerate}[label=(\roman*)]
\item If $m=m_0$ is fixed and $n\to\infty$, then
\[
x^B_{n,m_0}=n+O_{m_0}(n^{1/3}),\qquad A^B_{n,m_0}\asymp_{m_0}n^{-2/3},\qquad I_p(n,x^B_{n,m_0})\asymp_{p,m_0} n^{4/3-p/3}.
\]
\item There exists an integer $M_1\ge1$ such that, if $M_1\le m\le n$, then
\[
x^B_{n,m}-n\asymp n^{1/3}m^{2/3},\qquad A^B_{n,m}\asymp n^{-2/3}m^{1/6},
\]
\[
I_p(n,x^B_{n,m})\asymp_p
\begin{cases}
n^{4/3-p/3}m^{2/3-p/6},&1\le p<4,\\[4pt]
1+\log m,&p=4,\\[4pt]
n^{4/3-p/3},&p>4.
\end{cases}
\]
\item If $n\ge1$ and $m\ge n$, then
\[
x^B_{n,m}\asymp m,\qquad A^B_{n,m}\asymp (x^B_{n,m})^{-1/2},
\]
\[
I_p(n,x^B_{n,m})\asymp_p
\begin{cases}
(x^B_{n,m})^{2-p/2},&1\le p<4,\\[4pt]
1+\log x^B_{n,m},&p=4,\\[4pt]
n^{4/3-p/3},&p>4.
\end{cases}
\]
\end{enumerate}
\end{proposition}

\begin{proof}
(i) $x^B_{n,m_0}=n+O_{m_0}(n^{1/3})$ and $A^B_{n,m_0}\asymp_{m_0}n^{-2/3}$ are contained in Lemma \ref{lem:fixed-turning}. The upper bound for $I_p(n,x^B_{n,m_0})$ follows from Lemma \ref{lem:below-turning-lp}(i), since $x^B_{n,m_0}\le n+C_{m_0}n^{1/3}$ for all sufficiently large $n$. For the lower bound, choose a fixed $\delta>0$ smaller than the coefficient of $n^{1/3}$-term in Lemma \ref{lem:fixed-turning}. Then $[n,n+\delta n^{1/3}]\subset [0,x^B_{n,m_0}]$ for all large $n$, and the lower estimate in Lemma \ref{lem:below-turning-lp} applies.

(ii) Choose $M_1$ not smaller than the thresholds in Lemmas \ref{lem:zeros-near} and \ref{lem:normalizing-factors}. Lemma \ref{lem:zeros-near} gives
\[
x^B_{n,m}-n\asymp n^{1/3}m^{2/3}.
\]
In particular $x^B_{n,m}\asymp n$ and $(x^B_{n,m})^2-n^2\asymp n^{4/3}m^{2/3}$. Lemma \ref{lem:normalizing-factors} therefore gives $A^B_{n,m}\asymp n^{-2/3}m^{1/6}$. Let $S=x^B_{n,m}-n$. If $S\le n$, then $S/n^{1/3}\asymp m^{2/3}$, and Lemma \ref{lem:oscillatory-average} gives the stated estimates for $I_p(n,x^B_{n,m})$. If $S>n$, then the same relation forces $m\asymp n$. In that subcase $x^B_{n,m}\asymp n$ and Lemma \ref{lem:far-lp-integral} gives exactly the asymptotic stated in (ii) since $m\asymp n$.

(iii) The estimate $x^B_{n,m}\asymp m$ follows from Lemma \ref{lem:zeros-global}, after absorbing the finitely many small pairs into the constants. If $n$ belongs to a fixed finite set, then Lemma \ref{lem:fixed-order-estimates} gives $A^B_{n,m}\asymp (x^B_{n,m})^{-1/2}$ and the stated asymptotic for $I_p(n,x^B_{n,m})$ as $m\to\infty$. Hence it remains to consider $n$ large.

For large $n$ and $m\ge n$, Lemma \ref{lem:zeros-near} gives $x^B_{n,m}-n\ge c n$. Thus Lemma \ref{lem:far-lp-integral} applies with $R=x^B_{n,m}$ and gives the asymptotic for $I_p(n,x^B_{n,m})$. Moreover, since $x^B_{n,m}\ge (1+c)n$ we have $(x^B_{n,m})^2-n^2\asymp (x^B_{n,m})^2$. Lemma \ref{lem:normalizing-factors} then gives $A^B_{n,m}\asymp (x^B_{n,m})^{-1/2}$. This proves all asymptotics in (iii).
\end{proof}

\subsection{The asymptotic proof of Theorem \ref{thm:profile} for finite \texorpdfstring{$p$}{p}}

\begin{proof}[Proof of Theorem \ref{thm:profile} for $1\le p<\infty$]
It is enough to verify the asserted limit after passing to subsequences. Along any sequence either $n\to\infty$, or $n$ is bounded and hence is fixed after passing to a further subsequence. By the convention in the definition, the latter case belongs to the endpoint $\gamma=\infty$ and will be treated at the end of the proof. We first assume that $n\to\infty$ and write $\frac{\log m}{\log n}\to\gamma$.

If $\gamma=0$ and $m$ is bounded, then after passing to a subsequence $m$ is fixed. Proposition \ref{prop:lp-bessel-estimates}(i) gives
\[
\norm{u^B_{n,m}}_{L^p(\D)}^p\asymp_{p,m}(n^{-2/3})^{-p}n^{-2}n^{4/3-p/3}=n^{p/3-2/3}.
\]
Thus
\[
\frac{\log\norm{u^B_{n,m}}_{L^p(\D)}}{\log\lambda^B_{n,m}}\to\frac16-\frac1{3p}=\frac{p-2}{6p},
\]
which is $\Phi_p(0)$.

Assume next that $m\to\infty$. After passing to a further subsequence, we may assume either $m\le n$ or $m\ge n$. For $m\le n$, we must have $0\le\gamma\le1$. Then $x^B_{n,m}=n^{1+o(1)}$ and $\log\lambda^B_{n,m}=2\log n+o(\log n)$. By Lemma \ref{lem:lp-normalization} and Proposition \ref{prop:lp-bessel-estimates}, if $1\le p<4$,
\[
\norm{u^B_{n,m}}_{L^p(\D)}^p\asymp_p\left(\frac nm\right)^{(p-2)/3}.
\]
Therefore
\[
\frac{\log\norm{u^B_{n,m}}_{L^p(\D)}}{\log\lambda^B_{n,m}}\to\frac{(1-\gamma)(p-2)}{6p}.
\]
For $p=4$, the same computation has an additional factor $1+\log m$, which is sub-polynomial in $n$ and does not affect the logarithmic exponent. If $p>4$, then
\[
\norm{u^B_{n,m}}_{L^p(\D)}^p\asymp_p n^{p/3-2/3}m^{-p/6},
\]
and hence
\[
\frac{\log\norm{u^B_{n,m}}_{L^p(\D)}}{\log\lambda^B_{n,m}}\to\frac16-\frac1{3p}-\frac\gamma{12}.
\]

For $m\ge n$, we must have $1\le\gamma\le\infty$. We first consider $\gamma<\infty$. Lemma \ref{lem:zeros-global} gives
\[
x^B_{n,m}=n^{\gamma+o(1)},\qquad \log\lambda^B_{n,m}=2\gamma\log n+o(\log n).
\]
If $1\le p<4$, then Proposition \ref{prop:lp-bessel-estimates} gives
\[
\norm{u^B_{n,m}}_{L^p(\D)}^p\asymp_p 1,
\]
so the exponent is zero. If $p=4$, the norm grows at most like a power of $\log x^B_{n,m}$, and the exponent is again zero. If $p>4$, then
\[
\norm{u^B_{n,m}}_{L^p(\D)}^p\asymp_p (x^B_{n,m})^{p/2-2}n^{4/3-p/3}.
\]
Dividing the logarithm by $\log\lambda^B_{n,m}$ gives
\[
\frac{\log\norm{u^B_{n,m}}_{L^p(\D)}}{\log\lambda^B_{n,m}}\to\frac14-\frac1p-\frac1{6\gamma}+\frac{2}{3p\gamma}.
\]

If $\gamma=\infty$ and $n\to\infty$, Lemma \ref{lem:zeros-global} gives $x^B_{n,m}\asymp m$. In particular $\log n=o(\log x^B_{n,m})$ and $\log\lambda^B_{n,m}=2\log x^B_{n,m}+o(\log x^B_{n,m})$. If $1\le p<4$, Proposition \ref{prop:lp-bessel-estimates} gives $\norm{u^B_{n,m}}_{L^p(\D)}\asymp_p1$, and the exponent is zero. The same is true for $p=4$, since the growth is logarithmic. If $p>4$, then
\[
\norm{u^B_{n,m}}_{L^p(\D)}^p\asymp_p (x^B_{n,m})^{p/2-2}n^{4/3-p/3}.
\]
The factor involving $n$ is sub-polynomial in $x^B_{n,m}$, and hence
\[
\frac{\log\norm{u^B_{n,m}}_{L^p(\D)}}{\log\lambda^B_{n,m}}\to\frac14-\frac1p.
\]

Finally, if $n$ is fixed and $m\to\infty$, Lemma \ref{lem:fixed-order-estimates} and Lemma \ref{lem:lp-normalization} give
\[
\norm{u^B_{n,m}}_{L^p(\D)}\asymp_{p,n}
\begin{cases}
1,&1\le p<4,\\[2pt]
(1+\log x^B_{n,m})^{1/4},&p=4,\\[2pt]
(x^B_{n,m})^{1/2-2/p},&p>4.
\end{cases}
\]
This gives the value $\Phi_p(\infty)$ and completes the proof of Theorem \ref{thm:profile}.
\end{proof}

\subsection{The uniform \texorpdfstring{$L^p$}{Lp} estimates}

\begin{proof}[Proof of Theorem \ref{thm:uniform} for $1\le p<\infty$]
The range of large $n$ with $1\le n\le m$ is covered by Proposition \ref{prop:lp-bessel-estimates}(iii), while bounded $n$ are covered by Lemma \ref{lem:fixed-order-estimates}. For the range $1\le m\le n$, the finitely many small values of $m$ are covered by Proposition \ref{prop:lp-bessel-estimates}(i), and the large values of $m$ are covered by Proposition \ref{prop:lp-bessel-estimates}(ii). After changing constants once more to absorb the finitely many values of $(n,m)$, the following estimates hold uniformly.

If $1\le m\le n$, then
\begin{equation}\label{eq:uniform-proof-near-norms}
\norm{u^B_{n,m}}_{L^p(\D)}^p\asymp_p
\begin{cases}
\left(\dfrac nm\right)^{(p-2)/3},&1\le p<4,\\[8pt]
 n^{2/3}m^{-2/3}(1+\log m),&p=4,\\[4pt]
 n^{p/3-2/3}m^{-p/6},&p>4.
\end{cases}
\end{equation}
For $n\ge1$ and $m\ge n$, one has
\begin{equation}\label{eq:uniform-proof-far-norms}
\norm{u^B_{n,m}}_{L^p(\D)}^p\asymp_p
\begin{cases}
1,&1\le p<4,\\[4pt]
1+\log x^B_{n,m},&p=4,\\[4pt]
(x^B_{n,m})^{p/2-2}n^{4/3-p/3},&p>4.
\end{cases}
\end{equation}
Finally, for $n=0$
\begin{equation}\label{eq:uniform-proof-fixed-norms}
\norm{u^B_{n,m}}_{L^p(\D)}^p\asymp_p
\begin{cases}
1,&1\le p<4,\\[4pt]
1+\log x^B_{n,m},&p=4,\\[4pt]
(x^B_{n,m})^{p/2-2},&p>4.
\end{cases}
\end{equation}
These are exactly the estimates in the preceding results after substituting the normalizing factors into Lemma \ref{lem:lp-normalization}.

We prove the upper bound first. If $1\le p\le2$, H\"older's inequality and $\norm{u_\lambda}_{L^2(\D)}=1$ give
\[
\norm{u_\lambda}_{L^p(\mathbb D)}\le \pi^{1/p-1/2},
\]
so $b_p=0$ in this range. This is consistent with the above three estimates which also imply $\norm{u_\lambda}_{L^p(\mathbb D)}$ is bounded by a constant.

If $2<p<4$. From \eqref{eq:uniform-proof-near-norms}, when $m\le n$,
\[
\norm{u^B_{n,m}}_{L^p(\D)}^p\le C_p (x^B_{n,m})^{(p-2)/3}.
\]
The case $m\ge n\ge1$ and $n=0$ are bounded by a constant by \eqref{eq:uniform-proof-far-norms} and \eqref{eq:uniform-proof-fixed-norms}. Thus
\[
\norm{u^B_{n,m}}_{L^p(\D)}\le C_p(\lambda^B_{n,m})^{(p-2)/(6p)}.
\]
At $p=4$, the logarithmic factors are still harmless: if $m\le n$, then $m^{-2/3}(1+\log m)\le C$, while if $m\ge n\ge1$ or $n=0$, then $1+\log x^B_{n,m}\le C(x^B_{n,m})^{2/3}$. Hence the same exponent holds at $p=4$.

Let $4<p<8$. If $m\le n$, then \eqref{eq:uniform-proof-near-norms} gives
\[
\norm{u^B_{n,m}}_{L^p(\D)}^p\le C_p(x^B_{n,m})^{p/3-2/3}.
\]
If $m\ge n\ge1$, then the last line of \eqref{eq:uniform-proof-far-norms} gives
\[
\norm{u^B_{n,m}}_{L^p(\D)}^p\le C_p(x^B_{n,m})^{p/2-2}.
\]
The $n=0$ estimate gives the same bound. Since $p<8$, one has $p/2-2\le p/3-2/3$, and therefore all cases satisfy
\[
\norm{u^B_{n,m}}_{L^p(\D)}\le C_p(\lambda^B_{n,m})^{(p-2)/(6p)}.
\]

Finally suppose $p\ge8$. Then $p/2-2\ge p/3-2/3$. The same estimates as in the case $4<p<8$ imply
\[
\norm{u^B_{n,m}}_{L^p(\D)}^p\le C_p(x^B_{n,m})^{p/2-2}.
\]
Thus
\[
\norm{u^B_{n,m}}_{L^p(\D)}\le C_p(\lambda^B_{n,m})^{1/4-1/p}.
\]
This proves the upper bounds.

We now prove the lower bounds. If $1\le p<2$, interpolate between $L^p$ and $L^4$:
\[
\norm{u}_{L^2(\D)}\le \norm{u}_{L^p(\D)}^\theta \norm{u}_{L^4(\D)}^{1-\theta},\qquad
\frac12=\frac\theta p+\frac{1-\theta}{4}.
\]
Here $\theta=p/(4-p)$. Using the $L^4$ upper bound just proved, $\norm{u}_{L^4(\D)}\le C\lambda^{1/12}$, and $\norm{u}_{L^2(\D)}=1$, we get
\[
\norm{u}_{L^p(\D)}\ge c_p\lambda^{(p-2)/(6p)}.
\]
For $p=2$ the lower bound is the normalization. For $2<p\le4$, H\"older's inequality gives
\[
\norm{u_\lambda}_{L^p(\mathbb D)}\ge \pi^{1/p-1/2}.
\]
These lower bounds for $1\le p\le4$ can also be derived from the three estimates \eqref{eq:uniform-proof-near-norms}, \eqref{eq:uniform-proof-far-norms}, \eqref{eq:uniform-proof-fixed-norms} above.

It remains to consider $p>4$. If $m\le n$, then \eqref{eq:uniform-proof-near-norms} gives
\[
\norm{u^B_{n,m}}_{L^p(\D)}^p\ge c_p n^{p/3-2/3}m^{-p/6}\ge c_p(x^B_{n,m})^{(p-4)/6},
\]
where we used $m\le n$ and $x^B_{n,m}\asymp n$ in this range. If $m\ge n\ge1$, then \eqref{eq:uniform-proof-far-norms} gives
\[
\norm{u^B_{n,m}}_{L^p(\D)}^p\ge c_p(x^B_{n,m})^{p/2-2}n^{4/3-p/3}=c_p(x^B_{n,m})^{(p-4)/6}\left(\frac{x^B_{n,m}}{n}\right)^{(p-4)/3}\ge c_p(x^B_{n,m})^{(p-4)/6}.
\]
For $n=0$, \eqref{eq:uniform-proof-fixed-norms} gives the even stronger lower bound $c_p(x^B_{n,m})^{p/2-2}\ge c_p(x^B_{n,m})^{(p-4)/6}$. Since $\lambda^B_{n,m}=(x^B_{n,m})^2$, this proves
\[
\norm{u^B_{n,m}}_{L^p(\D)}\ge c_p(\lambda^B_{n,m})^{(p-4)/(12p)}
\]
for all $p>4$.

The sharpness of the lower bounds follows from fixed $m$ and $n\to\infty$ for $1\le p\le2$, fixed $n$ and $m\to\infty$ for $2\le p\le4$, and $m\asymp n\to\infty$ for $p>4$. The sharpness of the upper bounds follows from fixed $n$ and $m\to\infty$ for $1\le p\le2$, fixed $m$ and $n\to\infty$ for $2<p<8$, fixed $n$ and $m\to\infty$ again for $p\ge8$. This completes the proof of Theorem \ref{thm:uniform}.
\end{proof}

\section*{Acknowledgements}

The work is supported by National Key R\&D Program of China No. 2022YFA1007400 and National Natural Science Foundation of China No. 12525106. The author would like to express his sincere gratitude to his advisor, Professor Long Jin, for suggesting this problem and for many valuable comments, suggestions, and discussions concerning the results of this paper. The author would also like to thank Weiwei Wang and Zuoqin Wang for their suggestions regarding the question.

\bibliographystyle{amsalpha}
\bibliography{refs}

\end{document}